\documentclass[a4paper,10pt, oneside]{article}
\usepackage{amsmath,amssymb,amsthm,mathrsfs,graphicx}
\usepackage[affil-it]{authblk}
\usepackage{authblk}
\usepackage[inline]{enumitem}
\usepackage[font=small,labelfont=md,textfont=it]{caption}
\usepackage{floatrow,float}
\usepackage[titletoc, title]{appendix}
\usepackage[colorlinks,linkcolor=blue,citecolor=blue]{hyperref}
\usepackage{etoolbox}
\usepackage{longtable}
\usepackage{diagbox}
\usepackage{booktabs,makecell,multirow}
\usepackage[capitalise,nosort]{cleveref}
\usepackage{cases,color}
\usepackage{ulem}
\usepackage{pifont}
\usepackage{bm}
\usepackage{subfigure}
\crefname{equation}{}{}
\crefname{lemma}{Lemma}{Lemmas}
\crefname{theorem}{Theorem}{Theorems}
\crefname{discr}{Discretization}{Discretizations}
\crefname{assumption}{Assumption}{Assumptions}
\numberwithin{equation}{section}
\DeclareMathOperator{\A}{\mathbf{A}}
\DeclareMathOperator{\B}{\mathbf{B}}

\apptocmd{\sloppy}{\hbadness 10000\relax}{}{}

\newcommand{\ssnm}[1]
{
  \left\vert\kern-0.25ex
  \left\vert\kern-0.25ex
  \left\vert
  {#1}
  \right\vert\kern-0.25ex
  \right\vert\kern-0.25ex
  \right\vert
}

\makeatletter
\def\spher@harm#1{%
  \vbox{\hbox{%
    \offinterlineskip
    \valign{&\hb@xt@2\p@{\hss$##$\hss}\vskip.2ex\cr#1\crcr}%
  }\vskip-.36ex}%
}
\def\gshone{\spher@harm{.}}
\def\gshtwo{\spher@harm{.&.}}
\def\gshthree{\spher@harm{.&.&.}}
\let\gsh\spher@harm
\makeatother

\newtheorem{Assum}{Assumption}[section]
\newtheorem{lemma}{Lemma}[section]
\newtheorem{remark}{Remark}[section]
\newtheorem{theorem}{Theorem}[section]
\newtheorem{example}{Example}[section]

\makeatletter\def\@captype{table}\makeatother

\newlist{steps}{enumerate}{1}
\setlist[steps]{
	label=\textbf{Step  \arabic*.} ,
	align=left,
	leftmargin=!, 
	widest=Step 99., 
	itemindent=0pt,
	labelsep=0.5em,
	parsep=0.5\parskip,
	listparindent=0pt 
}

\begin{document}
	
\title{
  \Large\bf  Analysis of a fast fully discrete finite element method for fractional viscoelastic  wave propagation
\thanks
  {
    This work was supported in part by National Natural Science Foundation
    of China (12171340).
  }
}
\author{
  Hao Yuan \thanks{Email:  787023127@qq.com},
  Xiaoping Xie \thanks{Corresponding author. Email: xpxie@scu.edu.cn} \\
  {School of Mathematics, Sichuan University, Chengdu 610064, China}
}

\date{}
\maketitle

\begin{abstract}
	This paper is devoted to a numerical analysis of a  fractional viscoelastic  wave propagation model that generalizes   the  fractional  Maxwell model and  the  fractional  Zener model.
	First, we convert the    model  problem  into a velocity type integro-differential equation  and   establish   existence, uniqueness and regularity of its solution. Then   we consider a  conforming linear/bilinear/trilinear  finite element semi-discrete scheme and  a fast scheme of backward Euler full  discretization  with a sum-of-exponentials (SOE) approximation for the convolution integral, and   derive error estimates for the semi-discrete and fully discrete schemes. Finally, we provide several numerical examples to verify the theoretical results.
\end{abstract}

\medskip\noindent{\bf Keywords:} Fractional  viscoelasticity model;   wave propagation; integro-differential equation; finite element method; error estimate

\section{Introduction}
Let $\Omega\subset \mathbb{R}^d (d=2,\, 3)$ be a convex polygon domain with boundary $\partial\Omega$, and let $T$ be a positive constant. Consider the following general fractional viscoelasticity model of wave propagation:
\begin{equation}\label{model}
	\left\{
	\begin{array}{ll}
		\rho \frac{\partial^2\mathbf{u}}{\partial t^2}=\mathrm{\textbf{div}}\sigma+\mathbf{f}, & (x,t)\in\Omega\times[0,T], \\
		\sigma+\tau_\sigma^\alpha\frac{\partial^{\alpha}\sigma}{\partial t^{\alpha}}=\mathbb{C}\varepsilon(\mathbf{u})+\tau_\varepsilon^\alpha\mathbb{D}\frac{\partial^{\alpha}\varepsilon(\mathbf{u})}{\partial t^{\alpha}}, &(x,t)\in\Omega\times[0,T],\\
		\mathbf{u}=0, &  (x,t)\in\partial\Omega\times[0,T], \\
		\mathbf{u}(x,0)=\mathbf{u}_0(x),\frac{\partial\mathbf{u}}{\partial t}(x,0)=\mathbf{v}_0(x),\sigma(x,0)=\sigma_0(x),& x\in\Omega.
	\end{array}
	\right.
\end{equation}

\noindent Here  $\mathbf{u}$ is the displacement field, $\sigma=(\sigma_{ij})_{d\times d}$ is the symmetric stress tensor with its divergence given by $\textbf{div}\sigma=(\sum\limits_{i=1}^d \partial_i\sigma_{i1}, \cdots, \sum\limits_{i=1}^d \partial_i\sigma_{id} )^{\top}$, and $\varepsilon(\mathbf{u})=(\bigtriangledown \mathbf{u}+(\bigtriangledown \mathbf{u})^{\top})/2$ is the strain tensor. The parameters $\rho(>0)$, $\tau_\sigma(>0)$ and $\tau_\varepsilon(\geq 0)$ represent the mass density, the relaxation time and the retardation time, respectively. The fourth-order symmetric tensors $\mathbb{C}$ and $\mathbb{D}$ satisfy
 \begin{align}
	0 < M_0 \xi : \xi \leq \mathbb{C} \xi : \xi \leq M_1 \xi : \xi & \quad \forall \text{ symmetric tensor } \xi = (\xi_{ij})_{d \times d}, \text{ a.e. } x \in \Omega, \label{tensorC} \\
	0 < M_0 \xi : \xi \leq \mathbb{D} \xi : \xi \leq M_1 \xi : \xi & \quad \forall \text{ symmetric tensor } \xi = (\xi_{ij})_{d \times d}, \text{ a.e. } x \in \Omega, \label{tensorD}
\end{align}
where 
$M_0$ and $M_1$ are positive constants, and $\xi : \zeta = \sum_{i=1}^d \sum_{j=1}^d \xi_{ij} \zeta_{ij}$. 
Notice that for an isotropic elastic medium $\mathbb{C}\varepsilon(\mathbf{u})$ and $\mathbb{D}\varepsilon(\mathbf{u})$ take the forms 
\begin{align}\label{elastic-tensor}
	\mathbb{C}\varepsilon(\mathbf{u})=2\mu_C \varepsilon(\mathbf{u})+\lambda_C \text{div}\mathbf{u}  I, \qquad \mathbb{D}\varepsilon(\mathbf{u})=2\mu_D \varepsilon(\mathbf{u})+\lambda_D \text{div}\mathbf{u}  I,
\end{align}
respectively, where $\mu_C, \ \mu_D$,\ $\lambda_C$,\ $\lambda_D$ are the Lam\'e parameters, and $I$ the identity matrix. $\mathbf{f}$ is the body force and  $\mathbf{u}_0(x)$, $\mathbf{v}_0(x)$, $\sigma_0(x)$ are initial data.  For any function $\mathbf{w}(x,t)$ and $0<\alpha<1$,  let $\frac{\partial^{\alpha} \mathbf{w}}{\partial t^{\alpha}}$ be the $\alpha$-order  Caputo fractional derivative of $\mathbf{w}$    defined by
\begin{align}
	\frac{\partial^{\alpha} \mathbf{w}}{\partial t^{\alpha}}(t)=\frac{1}{\Gamma(1-\alpha)}\int_{0}^{t}(t-s)^{-\alpha}\frac{\partial\mathbf{w}}{\partial t}(s)\,ds.
\end{align}
It is worth noting that the classical viscoelasticity models correspond to specific choices of the relaxation and retardation times in the constitutive (second) equation of \eqref{model} with $\alpha = 1$:
\begin{itemize}
	\item The Maxwell model: $\tau_\sigma > 0$, $\tau_\varepsilon = 0$;
	\item The Zener model: $\tau_\sigma > 0$, $\tau_\varepsilon > 0$;
	\item The Kelvin-Voigt model: $\tau_\sigma = 0$, $\tau_\varepsilon > 0$.
\end{itemize}

Many materials exhibit both elastic and viscous kinematic behaviors simultaneously. This phenomenon, known as viscoelasticity, is typically characterized through the combination of ideal elements such as springs, which obey the Hooke's law, and viscous dashpots, which follow the Newton's law. Different configurations of these elements result in a variety of viscoelasticity models, including  the Maxwell model, the Zener model, and the Kelvin-Voigt model. For comprehensive treatments and applications of viscoelasticity theory, we refer the reader to \cite{1960Bland,2007Dill,1998Drozdov,  Fung1966International,1988Boundary,1962Gurtin,2000Nonlinear,Marques2012Computational} for several monographs on the development  and application  of viscoelasticity theory.

In recent decades, fractional-order differential operators, as an extension of integer-order ones, are widely used in many scientific and engineering fields, 
  due to their ability to accurately describe states or development processes that exhibit memory and hereditary characteristics. Specifically, in the context of viscoelastic materials with complex rheological properties, more and more research indicates that time-fractional viscoelasticity models  precisely characterize the dynamic behaviors of creep and relaxation, and effectively capture the effects of "fading" memory \cite{bagley1983theoretical,bagley1981fractional,   1986Bagley,blair1944analytical, caputo1971new, caputo1971linear, 2007Freed, gemant1936method, gemant1950frictional, mainardi2022fractional,rabotnov1970creep}. In the context of mathematical analysis, we refer to \cite{Larsson2010TheCG,oparnica2020well,SAEDPANAH2014201} for some wellposedness results of several   simplified fractional Zener models.

%
	

There has been considerable work in the literature on the numerical analysis of fractional viscoelasticity models. Enelund and Josefson \cite{enelund1997time} reformulated the constitutive equation of the fractional Zener model  as an integro-differential equation with a weakly singular convolution kernel, 
and conducted finite element simulations. Adolfsson et al. \cite{adolfsson2004adaptive} proposed a piecewise constant discontinuous Galerkin method for a fractional order  viscoelasticity differential constitutive equation. Subsequently, they \cite{adolfsson2008space} applied  a discontinuous Galerkin method in time and a continuous Galerkin finite element method in space to discretize an integro-differential equation modelling a quasi-static fractional viscoelasticity model. Larsson et al. \cite{Larsson2010TheCG} introduced a continuous Galerkin method for an integro-differential equation modelling a dynamic fractional order viscoelasticity model. They established a well-posedness result using semigroup theory, and derived the stability estimates and a priori error estimates. Later, they \cite{LARSSON2015196} applied a constant discontinuous Galerkin to the temporal discretization of the model and proved optimal order a priori error estimates.  Yu et al. \cite{yu2016fractional} employed finite element simulations for the fractional Zener model with an integro-differential form of constitutive equation in 3D cerebral arteries and aneurysms. Lam et al. \cite{lam2020exponential} presented a finite element scheme for the 1D fractional Zener model  with an integro-differential form of constitutive equation. In \cite{2024Liu-Xie}, Liu and Xie proposed a semi-discrete hybrid stress finite element method for a dynamic fractional viscoelasticity model and derived error estimates for the spatial semi-discrete scheme. Yuan and Xie \cite{yuan2024fastnumericalschemefractional} proposed a fast numerical scheme for the model \cref{model} with $\mathbb{C} = \mathbb{D}$.


In this paper, we carry out  a   finite element analysis for  the fractional viscoelasticity model \cref{model}. Our contributions are as follows:
\begin{itemize}
	\item 
	We   reformulate the fractional constitutive equation of model \cref{model} as   an integro-differential form to convert the original fractional viscoelasticity   model into a  velocity type integro-differential equation,  
	and   establish   existence, uniqueness and regularity results for  the resulting  model  based on the analytic semigroup theory and the Galerkin method.

	\item  We consider a  conforming  linear/bilinear/trilinear  finite element semi-discretization of the  resulting   model  and derive error estimates. 
	
	\item Based on the  spatial semi-discretization, we develop a fast full  discretization for the resulting model, where  the backward Euler scheme and a sum-of-exponentials (SOE) approximation are respectively  applied to the first-order temporal derivative of the velocity and  to   the convolution integral term  with  a convolution kernel of the Mittag-Leffler function.  We    derive error estimates for  the  proposed fully discrete scheme and provide numerical examples to test its performance. 
		
%
 \end{itemize}

The rest of this paper is arranged as follows. Section 2 introduces some preliminary results. Section 3 presents wellposedness  and regularity results for the resulting model. Section 4 carries out   error analysis for the semi-discrete scheme. Section 5 gives the fast fully discrete scheme and derives   error estimates. Finally, numerical examples are provided in Section 6 to verify the obtained theoretical results.

\section{Preliminaries}

\subsection{Notation}

For any non-negative integer $r$, let $H^r(\Omega)$ and $H^r_0(\Omega)$ denote the standard 
Sobolev spaces equipped with norm $\left\|\cdot\right\|_r$ and semi-norm $\left|\cdot\right|_r$, respectively. In particular, $H^0(\Omega) = L^2(\Omega)$ is the space of square-integrable functions, with  norm $\left\|\cdot\right\|:=\left\|\cdot\right\|_0$ and inner product $\langle\cdot,\cdot\rangle$. For simplicity, we use   boldface fonts to represent  the vector analogs of the Sobolev spaces, e.g. ${\bf H}^r(\Omega):=[H^r(\Omega)]^d.$ 

For any vector-valued  space ${\bf X}$ defined on $\Omega$ with norm $\left\|\cdot\right\|_{\bf X}$, we define
\[
L^p([0, T]; {\bf X}) := \left\{ \mathbf{w} : [0, T] \to {\bf X} \mid \left\|\mathbf{w}\right\|_{L^p({\bf X})} < \infty \right\},
\]
where
\[
\left\|\mathbf{w}\right\|_{L^p({\bf X})} := \left\{ \begin{array}{ll}
	\left( \displaystyle\int_{0}^{T} \left\|\mathbf{w}(t)\right\|_{\bf X}^p \, dt \right)^{1/p} & \text{if } 1 \leq p < \infty, \\
	\operatorname{esssup}_{0 \leq t \leq T} \left\|\mathbf{w}(t)\right\|_{\bf X} & \text{if } p = \infty.
\end{array} \right.
\]
 For simplicity,    we 
write $\mathbf{w}(t):=\mathbf{w}(x,t)$  and  $L^p({\bf X}):=L^p([0, T]; {\bf X})$. 
For an integer $r \geq 0$, the space $C^r({\bf X})$ can be defined analogously.   In subsequent analyses, ${\bf X}$ may be taken as ${\bf H}^r(\Omega)$   or ${\bf H}_0^r(\Omega)$. 


To facilitate the presentation, we use notation $C$ to denote a generic positive constant independent of the spatial mesh size $h$ and temporal mesh size $\Delta t$. Note that $C$ may depend on certain physics parameters of the model, such as $\alpha$, $T$, $\rho$, $M_0$, $M_1$, $\tau_\sigma$, $\tau_\varepsilon$, $\mu_C$, $\mu_D$, $\lambda_C$, $\lambda_D$, and may be different at its each occurrence.

\subsection{Equivalent  integro-differential equation and weak problem}
We note that  the constitutive equation  in the model \cref{model} is of the following   differential form:
\begin{equation}\label{constitutive}
	\sigma+\tau_\sigma^\alpha\frac{\partial^{\alpha}\sigma}{\partial t^{\alpha}}=\mathbb{C}\varepsilon(\mathbf{u})+\tau_\varepsilon^\alpha\mathbb{D}\frac{\partial^{\alpha}\varepsilon(\mathbf{u})}{\partial t^{\alpha}}.
\end{equation}
As shown  in \cite{yuan2024fastnumericalschemefractional},   applying  the Laplace transform and the inverse  Laplace transform to \eqref{constitutive} we obtain the following explicit expression of  $\sigma$ when $\tau_\sigma\neq 0$: 
\begin{equation}\label{sigma}
	\begin{aligned}
		\sigma=&\mathbb{C}\varepsilon(\frac{\partial\mathbf{u}}{\partial t})-\int_{0}^{t}E_{\alpha}(-(\frac{t-s}{\tau_\sigma})^\alpha)\left(\mathbb{C}-(\frac{\tau_\varepsilon}{\tau_\sigma})^\alpha\mathbb{D}\right)\varepsilon(\frac{\partial\mathbf{u}}{\partial t}(s)) \, ds+\iota_0,
	\end{aligned}
\end{equation}
where 
\begin{align}\label{iota-0}
	\iota_0:=E_{\alpha}(-(\frac{t}{\tau_\sigma})^\alpha)\big(\sigma_0-\mathbb{C}\varepsilon(\mathbf{u}_0)\big), 
\end{align} 
and   $E_{\alpha}(z)$ denotes the single parameter Mittag-Leffler function  defined by
\begin{align}\label{Mittag-Leffler}
	E_{\alpha}(z):=\sum_{j=0}^{\infty}\frac{z^j}{\Gamma(j\alpha+1)},\qquad \alpha>0.
\end{align}

\begin{remark}
For   $0<\alpha<1$ and $t\geq0$, there  holds (cf. \cite{BangtiJin2021}) 
    \begin{equation}\label{MLbound}
    	\frac{1}{1+\Gamma(1-\alpha)t^\alpha}\leq E_\alpha(-t^\alpha)\leq\frac{\Gamma(1+\alpha)}{\Gamma(1+\alpha)+t^\alpha}.
    \end{equation}
\end{remark}

Substitute the alternative constitutive equation \cref{sigma} into the first equation of \cref{model} and  introduce the velocity variable   
$$\mathbf{v}= \frac{\partial\mathbf{u}}{\partial t},$$
 then we get the following velocity type integro-differential equation of the fractional viscoelasticity model:
\begin{equation}\label{indiffvis}
	\left\{
	\begin{array}{ll}
		\frac{\partial\mathbf{v}}{\partial t}-\mathrm{\textbf{div}}\left(\mathbb{A}\varepsilon(\mathbf{v})-\displaystyle\int_{0}^{t}\beta(t-s)\mathbb{B}\varepsilon(\mathbf{v}(s)) \, ds\right)=\mathbf{F}, & (x,t)\in\Omega\times[0,T], \\
		\mathbf{v}=0,&  (x,t)\in\partial\Omega\times[0,T], \\
		\mathbf{v}(0)= \mathbf{v}_0,& x\in\Omega.
	\end{array}
	\right.
\end{equation}
where 
\begin{equation}\label{rewritenotation}
\left\{
\begin{aligned}
&\mathbb{A}:=\rho^{-1} \mathbb{C},\quad &&\mathbb{B}:= \rho^{-1}\left(\mathbb{C}-(\frac{\tau_\varepsilon}{\tau_\sigma})^\alpha\mathbb{D}\right),\\
&	\beta(t):=E_{\alpha}(-(\frac{t}{\tau_\sigma})^\alpha), \quad && \mathbf{F}:= \rho^{-1}\left(\mathbf{f}+\mathrm{\textbf{div}}\iota_0\right).
\end{aligned}
\right.
\end{equation}

We consider the following weak formulation of \cref{indiffvis}: 
For $t\in (0,T], $ find $\mathbf{v}(t) \in {\bf H}^1_0(\Omega) $ 
 such that

\begin{equation}\label{weak problem}
	\left\{
	\begin{array}{l}
		\langle\frac{\partial\mathbf{v}}{\partial t}, \mathbf{w}\rangle + a(\mathbf{v}, \mathbf{w}) - \displaystyle\int_{0}^{t} \beta(t - s) b(\mathbf{v}(s), \mathbf{w}) \, ds = \langle \mathbf{F}, \mathbf{w}\rangle, \quad \forall\ \mathbf{w} \in {\bf H}^1_0(\Omega), \\ 
		\mathbf{v}(0) = \mathbf{v}_0,
	\end{array}
	\right.
\end{equation}
where
\begin{equation}\label{a-b}
	\begin{aligned}
		&a(\mathbf{v}, \mathbf{w}):=\int_{\Omega}\mathbb{A}\varepsilon(\mathbf{v}):\varepsilon(\mathbf{w}) \, dx,\ \   
		 b(\mathbf{v}, \mathbf{w}):=\int_{\Omega}\mathbb{B}\varepsilon(\mathbf{v}):\varepsilon(\mathbf{w}) \, dx. 
	\end{aligned}
\end{equation}

 Under the assumption 
 \cref{tensorC}, we immediately have
\begin{equation}\label{a-property}
\left\{\begin{aligned}
	&a(\mathbf{v}, \mathbf{v}) \geq C\left\|\mathbf{v}\right\|_1^2,\\
	&a(\mathbf{v}, \mathbf{w})=a(\mathbf{v}, \mathbf{w}),
	\end{aligned}\right.
\quad 	\forall\  \mathbf{v}, \mathbf{w}\in {\bf H}_0^1(\Omega).
\end{equation}
 We also have the following boundedness results:
 \begin{equation}\label{bounded-property}
\left\{\begin{aligned}
	&a(\mathbf{v}, \mathbf{w}) \leq C\left\|\mathbf{v}\right\|_1\cdot\left\|\mathbf{w}\right\|_1,\\
	&b(\mathbf{v}, \mathbf{w})\leq C\left\|\mathbf{v}\right\|_1\cdot\left\|\mathbf{w}\right\|_1,
	\end{aligned}\right.
\quad 	\forall\  \mathbf{v}, \mathbf{w}\in {\bf H}_0^1(\Omega).
\end{equation}

\subsection{Basic  results}
\subsubsection{Analytic semigroup}
Define an  operator $\A: {\bf H}^2(\Omega)\cap {\bf H}_0^1(\Omega) \rightarrow {\bf L}^2(\Omega)$ by 
\begin{equation}\label{operatordef}
	\A  := -\mathrm{\textbf{div}}(\mathbb{A}\varepsilon(\mathbf{\cdot})). 
\end{equation}
From \eqref{a-property} we easily get
\begin{equation}\label{A-property}
\left\{\begin{aligned}
	&\langle \A\mathbf{v},\mathbf{v}\rangle=a(\mathbf{v}, \mathbf{v}) \geq C\left\|\mathbf{v}\right\|_1^2,\\
	& \langle \A\mathbf{v},\mathbf{w}\rangle=a(\mathbf{v}, \mathbf{w})=\langle\mathbf{v},\A\mathbf{w}\rangle.
	\end{aligned}\right.\quad 
	\forall\ \mathbf{v}, \mathbf{w}\in {\bf H}^2(\Omega)\cap {\bf H}_0^1(\Omega).
\end{equation}
We  mention that  the following  inequality (cf. \cite{Brenner2008})  also holds:
	\begin{equation}\label{regu-ineq}
		\left\|\mathbf{v}\right\|_2\leq C\left\|\A\mathbf{v}\right\|,\qquad \forall\ \mathbf{v}\in {\bf H}^2(\Omega)\cap {\bf H}_0^1(\Omega).
\end{equation}

By definition and   \eqref{A-property} we see  that $\A$ is a linear, positive, self-adjoint elliptic operator.
 Thus, From   \cite[Theorem 1 in section 6.5.1]{Evans2010}   there exists a complete orthonormal basis,
$$\left\{\bm{\varphi}_k\right\}_{k=1}^{\infty}\subset {\bf H}^2(\Omega)\cap {\bf H}_0^1(\Omega),$$
 of ${\bf L}^2(\Omega)$ and a nondecreasing sequence $\left\{\lambda_k>0\right\}_{k=1}^{\infty}$ such that
\begin{equation}\label{eigenA}
\A \bm{\varphi}_k = \lambda_k \bm{\varphi}_k,\quad k=1,2,\cdots.
\end{equation}
Meanwhile, $\left\{\lambda^{-1/2}\bm{\varphi}_k\right\}_{k=1}^{\infty}$  is an orthonormal basis of $ {\bf H}_0^1(\Omega)$  equipped with the
inner product $a(\cdot,\cdot)$.



Denote by
\begin{align}\label{semigroup} 
S(t) := \exp(-t\A) 
\end{align}
 the analytic semigroup generated by \( -\A \). 
 The following lemma  shows some semigroup properties of $S(t)$.

\begin{lemma}\label{fuzhuyinli3}
	\cite{Larsson1998,Thomee2006} For   \( 0 < \alpha < 1 \),   let \( g \in C^1({\bf H}_0^1(\Omega)) \)
	 satisfy $$ \left\|g(t)\right\| \leq C t^{-\alpha} , \quad  \left\|\frac{\partial g}{\partial t}(t)\right\| \leq C t^{-\alpha-1}, \quad \forall\ t\in (0,T] .$$ Then 
	\begin{equation}
		\left\|\A\int_{0}^{t} S(t-s) g(s) \, ds\right\| \leq C t^{-\alpha}, \quad \forall\ t\in (0,T] .
	\end{equation}
	Furthermore, 
	for any non-negative integer $k$ and $ t\in (0,T]$ there hold
	\begin{equation}\label{sgregularity}
		\left\{\begin{aligned}
			&\left\|\frac{\partial^kS(t)}{\partial t^k} \mathbf{v}\right\|_2 \leq C t^{-k} \left\|\mathbf{v}\right\|_2, \quad & \forall \ &\mathbf{v} \in {\bf H}^2(\Omega)\cap {\bf H}_0^1(\Omega),\\
			&\left\|\frac{\partial S(t)}{\partial t} \mathbf{v}\right\|\leq C t^{-\frac{1}{2}}\left\|\mathbf{v}\right\|_1, \quad & \forall \ &\mathbf{v} \in {\bf H}^1_0(\Omega),\\
			&\left\|\frac{\partial^kS(t)}{\partial t^k} \mathbf{v}\right\|_{2j} \leq C t^{-k-j} \left\|\mathbf{v}\right\|, \quad & \forall \ &\mathbf{v} \in {\bf L}^2(\Omega), \   \ j = 0, 1.
		\end{aligned}
		\right.
	\end{equation}
\end{lemma}

\begin{lemma}
	For $T>0$ and $0<\alpha<1$, there holds
	\begin{equation}\label{ML-t-bound}
		\left|\frac{d E_\alpha(-t^\alpha)}{dt}\right|\leq Ct^{-\alpha-1},\qquad \forall\ t\in(0,T].
	\end{equation}
\end{lemma}
\begin{proof}
	Due to 
	\begin{align*}
		\left|\frac{d E_\alpha(-t^\alpha)}{dt}\right|=t^{\alpha-1}E_{\alpha,\alpha}(-t^\alpha).
	\end{align*}
	Note that $t^{2\alpha}E_{\alpha,\alpha}(-t^\alpha)$ is bounded in $[0,T]$. Then, there exists a constant $C$ such that for $0\leq t\leq T$
	\begin{align*}
		t^{2\alpha}E_{\alpha,\alpha}(-t^\alpha)\leq C,
	\end{align*}
	which indicates
	\begin{align*}
		\left|\frac{d E_\alpha(-t^\alpha)}{dt}\right|=t^{-\alpha-1}\cdot t^{2\alpha}E_{\alpha,\alpha}(-t^\alpha)\leq Ct^{-\alpha-1}.
	\end{align*}
\end{proof}


\subsubsection{Volterra integro-differential system}

Let  $ {m}$ be a positive integer. Consider the Volterra integro-differential  system 
 \begin{equation}\label{volterraintegral-eq}
		\mathbf{y}(t)-\int_{0}^{t}\mathbf{G}(t-s)\mathbf{y}(s) \, ds=\mathbf{g}(t), \qquad t\in[0,T],
	\end{equation}
where  $\mathbf{g}(t)=\left(g_1(t),g_2(t),\cdots,g_{ {m}}(t)\right)^{\top}$ is a  vector-valued function  and $\mathbf{G}(t)=(G_{ij}(t))_{i,j=1}^{ {m}}$ is a matrix-valued function.  We have the following lemma:
\begin{lemma}\label{solutioncontinuous}
	\cite{Brunner_2017} 
Suppose that the functions $\mathbf{g}$ and $\mathbf{G}$ satisfy	$$ g_i,  G_{ij}\in C^0[0,T]\cap C^1(0,T].$$
	 Then the  system \cref{volterraintegral-eq}
		admits a unique   solution 
	$\mathbf{y}(t)=\left(y_1(t),y_2(t),\cdots,y_{ {m}}(t)\right)^{\top}$ with $y_i\in C^0[0,T]\cap C^1(0,T] $. 
\end{lemma}

Based on \cref{solutioncontinuous} and a standard  technique, 
we can obtain the following result:

\begin{lemma}\label{solutionofintegrodifferential}
Let  $ M_{\gamma_1},M_{\gamma_2}$ be two $ {m}\times {m}$ matrices. Under the same conditions as in Lemma \ref{solutioncontinuous},  the  system 
 \begin{equation}\label{volterraintediff-eq}
  \left\{
  \begin{aligned}
   &\frac{\partial \mathbf{z}}{\partial t}(t)+M_{\gamma_1}\mathbf{z}(t)-M_{\gamma_2}\int_{0}^{t}\mathbf{G}(t-s)\mathbf{z}(s) \, ds=\mathbf{g}(t), \qquad t\in(0,T], \\
   &\mathbf{z}(0) = \mathbf{z}_0
  \end{aligned}
  \right.
 \end{equation} 
 admits a unique solution $\mathbf{z}(t)=\left(z_1(t),z_2(t),\cdots,z_{ {m}}(t)\right)^{\top}$ with $\mathbf{z}_i\in C^1[0,T]\cap C^2(0,T]$.
\end{lemma}
\begin{proof}
 Denote $\mathbf{y}=\frac{\partial \mathbf{z}}{\partial t}$, then we have
  \begin{align}\label{212}
   \mathbf{z}(t)=\mathbf{z}_0+\int_{0}^{t}\mathbf{y}(\tau) \, d\tau.
  \end{align}
  Rewriting \cref{volterraintediff-eq} in terms of $\mathbf{y}$, we get
  \begin{equation}\label{213}
   \mathbf{y}(t)-\int_{0}^{t}\tilde{\mathbf{G}}(t-s)\mathbf{y}(s) \, ds=\tilde{\mathbf{g}}(t),
  \end{equation}
  where
  \begin{align*}
   &\tilde{\mathbf{G}}(t)=-M_{\gamma_1}+\int_{0}^{t}\mathbf{G}(\tau) \, d\tau,\quad
   \tilde{\mathbf{g}}(t)=\mathbf{g}(t)-M_{\gamma_1}\mathbf{z}_0-M_{\gamma_2}\mathbf{z}_0\int_{0}^{t}\mathbf{G}(t-\tau) \, d\tau.
  \end{align*}
  By \cref{solutioncontinuous} we know that the equation \eqref{213} admits a unique solution $\mathbf{y}(t)=\left(y_1(t),y_2(t),\cdots,y_{ {m}}(t)\right)^{\top}$ with $y_i\in C^0[0,T]\cap C^1(0,T] $.
  As a result, the desired conclusion follows from \eqref{212}.
  \end{proof}

\subsubsection{Several   inequalities}

Let us first introduce two types of Gronwall's inequalities (cf. \cite{Thomee2006}):

\begin{lemma}{\bf{(Continuous Gronwall's inequality)}}
	\label{fuzhuyinli2}
	 Let   \( 0 < \alpha < 1 \) and let \( y(t) \) and \( b(t) \) be two nonnegative functions. Suppose that
	\begin{equation}
		y(t) \leq b(t) + \gamma \int_{0}^{t} (t-s)^{-\alpha} y(s) \, ds, \qquad \forall\ t \in [0, T]
	\end{equation}
	with \( \gamma \geq 0 \), then there exists a constant $C>0$ such that
	\begin{equation}
		y(t) \leq b(t) + C \int_{0}^{t} (t-s)^{-\alpha} b(s) \, ds, \qquad \forall\ t \in [0, T].
	\end{equation}
	Furthermore, if \( b(t) \) is a nondecreasing function of \( t \), then   	\begin{equation}
		y(t) \leq C b(t), \qquad \forall\ t \in [0, T].
	\end{equation}
\end{lemma}

 \begin{lemma}{\bf{(Discrete Gronwall's inequality)}}
 \label{discretegronwall}
 Suppose that the sequence $(y_n)_{n\geq0}$ is non-negative and satisfy
 \begin{equation}
 	y_n\leq h_n+\sum_{i=0}^{n-1}\mu_{n-i}y_i, \quad \text{for}\ n\geq0,
 \end{equation}
 where $ (h_n)_{n\geq0}$ is a nondecreasing sequence and $\mu_{i}\geq0$ for any $i$, then 
 \begin{equation}
 	y_n\leq h_n \exp(\sum_{i=0}^{n-1}\mu_{n-i}).
 \end{equation}
 \end{lemma}

 Next we introduce  two auxiliary lemmas (cf. \cite{Chen1992}):
 
 \begin{lemma}\label{fuzhuyinli-continuous}
 Let  $G\in L^1(0,T)$ and $f\in L^2(0,T)$. Then for any $r\in (0,T]$ and $\varepsilon>0$, there exists a constant $C$ depending only on $\varepsilon$ such that
 \begin{equation}
 	\left|\int_{0}^{r}\int_{0}^{t}G(t-s)f(s)f(t)\, ds\, dt\right|\leq\varepsilon\int_{0}^{r}f^2(t)\, dt+C\int_{0}^{r}\left|G(r-t)\right|\int_{0}^{t}f^2(s)\, ds\, dt.
 \end{equation}
 \end{lemma}
 \begin{lemma}\label{fuzhuyinlikappa}
 Let $G\in L^1(0,T)$ and denote \begin{equation}
 	g_k:=\int_{t_{k-1}}^{t_k}G(s)\, ds,\quad \text{for}\ k=1, 2, \cdots, N,
 \end{equation}
 where $t_i=\frac{iT}N$ for $i=0, 1, \cdots, N$. 
 Then for $(f_k)_{k=0}^N$ and  any $\varepsilon>0$, there exists a constant $C$ depending on $T$, $G$ and $\varepsilon$ such that for $1\leq n\leq N$,
 \begin{equation}
 	\left|\sum_{k=1}^{n}\sum_{i=0}^{k-1}g_{k-i}f_if_k\right|\leq\varepsilon\sum_{k=1}^{n}f_k^2+C\sum_{k=0}^{n-1}|g_{n-k}|\sum_{i=0}^{k}f_i^2.
 \end{equation}
 \end{lemma}

\subsection{SOE approximation of Mittag-Leffler function}
\label{SOE-approx}

Notice that  in the weak formulation \eqref{weak problem}  there is a  convolution integral term with the convolution kernel   
 $$\beta(t-s)=E_{\alpha}(-(\frac{t-s}{\tau_{\sigma}})^\alpha).$$ 
As the Mittag-Leffler function
$E_{\alpha}(t)$
 is an infinite  series,   how to compute   $E_{\alpha}(-t^\alpha)$ efficiently is crucial  to the design of a fast numerical scheme  for the fractional viscoelastic model  \eqref{model}.
 In what follows we introduce  an SOE approximation of $E_{\alpha}(-t^\alpha)$  given in \cite{yuan2024fastnumericalschemefractional}. 
 

 Notice that   for $t>0$ and  $0<\alpha<1$  there holds \cite{2014Rogosin,meng2018green}
\begin{equation}\label{identityMLF}
		\begin{aligned}
			E_{\alpha}(-t^{\alpha})&=\frac{\sin(\alpha\pi)}{\pi}\int_{0}^{\infty}\frac{s^{\alpha-1}}{s^{2\alpha}+2s^\alpha\cos\alpha\pi+1}e^{-st}ds. 
		\end{aligned}
	\end{equation}
Applying  the integration variable substitution $x=s^{-\alpha}$ to \eqref{identityMLF},  we  get
\begin{equation}\label{identityMLF1}
		\begin{aligned}
			E_{\alpha}(-t^{\alpha})&= \int_{0}^{\infty}f(x,t,\alpha)dx=\left(\int_{0}^{1}+\int_{1}^{q^{1}}+...+\int_{q^{K-1}}^{q^{K}}+\int_{q^K}^{\infty}\right)f(x,t,\alpha)dx \\	
		&=\sum_{k=0}^{K}\int_{c_k-r_k}^{c_k+r_k}f(x,t,\alpha)dx+\int_{q^K}^{\infty}f(x,t,\alpha)dx,
		\end{aligned}
	\end{equation}
where
\begin{equation}
	f(x,t,\alpha):=\frac{\sin(\alpha\pi)}{\alpha\pi} \frac{e^{-tx^{-\frac 1\alpha} }}{x^{2}+2x\cos\alpha\pi+1},
	\nonumber
\end{equation}
 $q>1$ is a constant,   $K>0$ is an integer, and $c_k$ and $r_k$ are given by
\begin{equation}\label{ckrk}
	c_0=r_0=\frac{1}{2}, \quad
	c_k 
	=\frac{(q+1)q^{k-1}}{2}, \quad r_k=\frac{(q-1)q^{k-1}}{2},   \qquad k=1,2,\cdots,K.
\end{equation}

We apply the integration variable substitution $x=r_ky+c_k$ to each term $\displaystyle\int_{c_k-r_k}^{c_k+r_k}f(x,t,\alpha)dx$ and  obtain
\begin{equation}\label{2.19}
	\begin{aligned}
		\int_{c_k-r_k}^{c_k+r_k}f(x,t,\alpha)dx&
		=\int_{-1}^{1}g_k(y,t)dy, 
	\end{aligned}
\end{equation}
where \begin{equation*}\label{gk}
		g_k(y,t):=\frac{\sin(\alpha\pi)}{\alpha\pi} \frac{r_ke^{-t(r_ky+c_k)^{-\frac 1\alpha}}}{  (r_ky+c_k)^{2 }+2(r_ky+c_k)\cos\alpha\pi +1}, \quad k=0,1,  \cdots, K.
\end{equation*}
Consider the following Gaussian quadrature of   $g_k(x)$:
\begin{equation}\label{Gaussquadrature}
	\int_{-1}^{1}g_k(x)dx=\sum_{j=1}^{J}\omega_jg_k(\xi_j)+R_J(g_k),
\end{equation}
where $\omega_j$ and $\xi_j$ denote the Gaussian quadrature weights and nodes, respectively, and $R_J(g)$ denotes the remaining term. %
 
Substitute \eqref{Gaussquadrature} into \eqref{2.19}, and we get 
\begin{equation}\label{2.37}
	\begin{aligned}
		\int_{c_k-r_k}^{c_k+r_k}f(x,t,\alpha)dx&
		&= \sum_{j=1}^{J}b_{kj}e^{-ta_{kj}}+R_J(g_k),
	\end{aligned}
\end{equation}
where 
$$a_{kj}=(r_k\xi_j+c_k)^{-\frac{1}{\alpha}} , \quad 
 b_{kj}=\frac{\sin(\alpha\pi)}{\alpha\pi}\cdot \frac{\omega_jr_k}{  (r_k\xi_j+c_k)^{2 }+2(r_k\xi_j+c_k)\cos\alpha\pi +1} .$$

Substituting 
 \eqref{2.37} into \eqref{identityMLF1},   we finally get the   
SOE approximation of the Mittag-Leffler function: 
	\begin{align}\label{SOEapp}
		\beta(t)=  E_{\alpha}(-t^\alpha)&
		 =\sum_{k=0}^{K}\sum_{j=1}^{J}b_{kj}e^{-ta_{kj} }+R_{soe}(t)\nonumber\\
		 &=\sum_{j=1}^{N_{exp}}b_je^{-a_jt}+R_{soe}(t), 
	\end{align}
where  
	\begin{align}
		&R_{soe}(t):=\sum_{k=0}^{K}R_J(g_k)+\int_{q^K}^{\infty}f(x,t,\alpha)dx, \label{remain}
	\end{align}
is the truncation error,   $N_{exp}=(K+1)J$, and 
 $a_j$ and $b_j$ are the j-th elements of 
\begin{align*}
	&[a_{01},a_{02},...,a_{0J},a_{11},a_{12},...,a_{1J},...,a_{(K+1)1},...,a_{(K+1)(J-1)},a_{(K+1)J}]
\end{align*}
and 
\begin{align*}
	&[b_{01},b_{02},...,b_{0J},b_{11},b_{12},...,b_{1J},...,b_{(K+1)1},...,b_{(K+1)(J-1)},b_{(K+1)J}],
\end{align*}
respectively.

\begin{remark}\label{rmk2.4} 
As shown in \cite{yuan2024fastnumericalschemefractional}, for any tolerance error $0<\epsilon<1$ there holds 
\begin{align}\label{SOEapp2}
		\left|R_{soe}(t)\right| \leq \mathcal{O}(\epsilon),\quad 0<t\leq T,
	\end{align}
  provided that
$	K=\mathcal{O}(|\log\epsilon|)$ and $ J=\mathcal{O}(|\log(\epsilon^{-1}|\log\epsilon|)|).
$
Then the computation  complexity of the SOE approximation 
$E_{\alpha}(-t^\alpha)\approx\sum_{j=1}^{N_{exp}}b_je^{-a_jt} 
$
is $$ N_{exp}=(K+1)J= \mathcal{O}(|\log\epsilon|^2).$$
In actual computation we may  choose $\epsilon=\mathcal{O}(\triangle t)$, where   $\triangle t:=\frac{T}{N}<1$ is the temporal step size and $N$ is a positive integer. In this case we have  
	\begin{equation}
	\left|R_{soe}(t)\right| \leq \mathcal{O}(\triangle t), \quad 	N_{exp}  =\mathcal{O}(\log^2N).
	\end{equation}
\end{remark}

\section{Results of existence, uniqueness and regularity}

In this section we apply   the Galerkin method (cf. \cite{Evans2010}) to show   the existence and regularity of the solution to the weak problem \eqref{weak problem}. 
We first make the following assumptions on the functional $F$ and 
the initial velocity $\mathbf{v}_0$: 
\begin{Assum} \label{regularitycond0}
Suppose that
 \begin{equation}
		\left\{
		\begin{aligned}
			&\mathbf{F}(t) \in C^1({\bf L}^2(\Omega)),\ 
			\frac{\partial^2 F}{\partial t^2}(t)\in L^1({\bf L}^2(\Omega)), \\
 		& \mathbf{v}_0 \in   {\bf H}^2(\Omega)\cap {\bf H}_0^1(\Omega) ,\\
		&  \A\mathbf{v}_0+\mathbf{F}(0) \in {\bf H}_0^1(\Omega).
		\end{aligned}
		\right.
	\end{equation}

\end{Assum}

  \subsection{Galerkin approximations and stability results }

Recall that the operator $\A$, defined in \cref{operatordef}, admits a set of eigenfunctions 
$$\left\{\bm{\varphi}_k\right\}_{k=1}^{\infty}\subset {\bf H}^2(\Omega)\cap {\bf H}_0^1(\Omega)$$
that forms a complete orthonormal basis of ${\bf L}^2(\Omega)$. 
For   integer $m\geq 1$,   we set 
 $$\mathscr{P}_m:=\text{span }\left\{\bm{\varphi}_k:\ k=1, 2, \cdots, m\right\}.$$
  
  Let 
  $\mathbf{d}(t)=\left(d_1(t),d_2(t),...,d_m(t)\right)^{\top}$   solve
\begin{equation}\label{appcoe1}
\left\{\begin{array}{l} \frac{\partial \mathbf{d}}{\partial t}(t) + M_{\gamma_1} \mathbf{d}(t) - M_{\gamma_2} \displaystyle\int_{0}^{t} \mathbf{G}(t-s) \mathbf{d}(s) \, ds = \mathbf{g}(t), \\
\mathbf{d}(0) = \left(\langle\mathbf{v}_0, \bm{\varphi}_1\rangle,\langle\mathbf{v}_0, \bm{\varphi}_2\rangle,...,\langle\mathbf{v}_0, \bm{\varphi}_m\rangle\right)^{\top}.
\end{array}
\right.
\end{equation}
Here
\begin{align}\nonumber
 &M_{\gamma_1}=\begin{bmatrix}
  \begin{smallmatrix}
   a(\bm{\varphi}_1, \bm{\varphi}_1) & 0  & \cdots & 0  \\
   0 & a(\bm{\varphi}_2, \bm{\varphi}_2)  & \cdots & 0 \\
   \vdots & \vdots & \ddots & \vdots \\
   0 & 0  & \cdots & a(\bm{\varphi}_m, \bm{\varphi}_m)
 \end{smallmatrix} 
\end{bmatrix}
 ,\quad 
 \mathbf{G}(t)=\begin{bmatrix}
  \begin{smallmatrix}
   \beta(t) & 0  & \cdots & 0  \\
   0 & \beta(t)  & \cdots & 0 \\
   \vdots & \vdots & \ddots & \vdots \\
   0 & 0  & \cdots & \beta(t)
  \end{smallmatrix}
 \end{bmatrix},
 \\
 \nonumber
 &M_{\gamma_2}=\begin{bmatrix}
  \begin{smallmatrix}
   b(\bm{\varphi}_1, \bm{\varphi}_1) & b(\bm{\varphi}_2, \bm{\varphi}_1)  & \cdots & b(\bm{\varphi}_m, \bm{\varphi}_1)  \\
   b(\bm{\varphi}_1, \bm{\varphi}_2) & b(\bm{\varphi}_2, \bm{\varphi}_2)  & \cdots & b(\bm{\varphi}_m, \bm{\varphi}_2) \\
   \vdots & \vdots & \ddots & \vdots \\
   b(\bm{\varphi}_1, \bm{\varphi}_m) & b(\bm{\varphi}_2, \bm{\varphi}_m)  & \cdots & b(\bm{\varphi}_m, \bm{\varphi}_m)
  \end{smallmatrix}
 \end{bmatrix}
,\quad 
\mathbf{g}(t)=\begin{bmatrix}
 \begin{smallmatrix}
  \langle F, \bm{\varphi}_1\rangle \\
  \langle F, \bm{\varphi}_2\rangle \\
  \vdots \\
  \langle F, \bm{\varphi}_m\rangle
 \end{smallmatrix}
\end{bmatrix}.
\end{align}
In view of   \cref{solutionofintegrodifferential} and the assumption $\mathbf{F}(t) \in C^1({\bf L}^2(\Omega))$ in \cref{regularitycond0}
we easily get the existence and uniqueness of $\mathbf{d}(t)$, with 
$d_k(t)\in C^1[0,T]\cap C^2(0,T]$ for $k=1, 2, \cdots, m$.

 Denote
    \begin{equation}\label{operatordefB}
	 \B  := -\mathrm{\textbf{div}}(\mathbb{B}\varepsilon(\cdot))  
\end{equation}
and let $P_m$ denote the $L^2$-projection onto $\mathscr{P}_m$. 
Set 
\begin{equation}\label{appsolution}
	\mathbf{v}_m(t) := \sum_{k=1}^{m} d_k(t) \bm{\varphi}_k \in \mathscr{P}_m,
\end{equation}
 then from   \cref{appcoe1}, \eqref{operatordef} and  \eqref{a-b} 
we know that  $\mathbf{v}_m$ satisfies
\begin{equation}\label{appcoe}
	\left\{
	\begin{aligned}
		&\langle\frac{\partial\mathbf{v}_m}{\partial t}, \bm{\varphi}_k\rangle +\langle \A\mathbf{v}_m , \bm{\varphi}_k\rangle- \int_{0}^{t} \beta(t - s) \langle \B \mathbf{v}_m(s), \bm{\varphi}_k\rangle \, ds = \langle \mathbf{F}, \bm{\varphi}_k\rangle,\\
		&\langle\mathbf{v}_m(0), \bm{\varphi}_k\rangle = \langle\mathbf{v}_0, \bm{\varphi}_k\rangle
	\end{aligned}
	\right.
\end{equation}
for $   k=1,2,\cdots,m,$
or equivalently, 
\begin{equation}\label{appcoe3}
	\left\{
	\begin{aligned}
		&\frac{\partial\mathbf{v}_m}{\partial t} + \A\mathbf{v}_m - \int_{0}^{t} \beta(t - s) P_m \B \mathbf{v}_m(s) \, ds = P_m \mathbf{F}, \\
		&\mathbf{v}_m(0) =   P_m \mathbf{v}_0.
	\end{aligned}
	\right.
\end{equation}
 As a result, we  have the following result:
%
%

\begin{lemma}\label{volterra-in-di-existence}
	Assume that $\mathbf{F}(t) \in C^1({\bf L}^2(\Omega))$, 
then the  system 	\cref{appcoe3} admits a unique  solution  $\mathbf{v}_m\in C^1\left({\bf H}^1_0(\Omega)\right)\cap C^2\left({\bf H}^1_0(\Omega)\right)$.
\end{lemma}

For the  $L^2$-projection $P_m$,   we have following lemma:

\begin{lemma}\label{Pm-v0}  
 There hold
	\begin{equation}\label{L2re}
\left\{\begin{aligned}
&		  \left\|P_m\mathbf{v}\right \|\leq \left\|\mathbf{v}\right \|, &&\forall\ \mathbf{v} \in {\bf L}^2(\Omega);\\
&  \left\|P_m\mathbf{v}\right \|_2\leq C\left\|\mathbf{v}\right \|_2, &&\forall\ \mathbf{v} \in {\bf H}^2(\Omega) \cap {\bf H}^1_0(\Omega).
\end{aligned}\right.	
		\end{equation}	
 \end{lemma}
\begin{proof} The first inequality is  standard for  the $L^2$-projection $P_m$, so we only prove the second one. 

Since 
$\left\{\bm{\varphi}_k\right\}_{k=1}^{\infty}\subset {\bf H}^2(\Omega) \cap {\bf H}^1_0(\Omega)$
 forms a complete orthonormal basis of ${\bf L}^2(\Omega)$,  we have, 
for any 
$\mathbf{v} \in {\bf H}^2(\Omega) \cap {\bf H}^1_0(\Omega)$, 
	\begin{align}\label{v-expansion}
		\mathbf{v}= \sum_{k=1}^{\infty} \langle \mathbf{v}, \bm{\varphi}_k \rangle \bm{\varphi}_k, \quad
		P_m\mathbf{v} = \sum_{k=1}^{m} \langle \mathbf{v}, \bm{\varphi}_k \rangle \bm{\varphi}_k. 
	\end{align}
 Then from \eqref{eigenA} it follows
 \begin{align*}
		\left\|\A \mathbf{v}\right\|^2 
		&=\langle \sum_{k=1}^{\infty}\langle  \mathbf{v}, \bm{\varphi}_k \rangle  \A \bm{\varphi}_k, \sum_{k=1}^{\infty}\langle  \mathbf{v}, \bm{\varphi}_k \rangle  \A \bm{\varphi}_k\rangle\\
		&=\sum_{k=1}^{\infty} |\langle \mathbf{v}, \bm{\varphi}_k \rangle|^2  \lambda_k^2\langle \bm{\varphi}_k,   \bm{\varphi}_k \rangle\\
		&= \sum_{k=1}^{\infty} \lambda_k^2 |\langle \mathbf{v}, \bm{\varphi}_k \rangle|^2\\
		&\geq  \sum_{k=1}^{m} \lambda_k^2 |\langle \mathbf{v}, \bm{\varphi}_k \rangle|^2
		=\left\|\A P_m\mathbf{v}\right\|^2.		
	\end{align*}
By the definition of $\A$  and its property \cref{regu-ineq} we easily know
	\begin{equation}\label{H2-A-equiv}
	C \left\|\A\mathbf{v}\right\|	\leq\left\|\mathbf{v}\right\|_{2} \leq C \left\|\A\mathbf{v}\right\|
	\end{equation}
	and 	\begin{equation*}
	  C \left\|\A P_m\mathbf{v}\right\|	\leq\left\|P_m\mathbf{v}\right\|_{2} \leq C \left\|\A P_m\mathbf{v}\right\|.
	\end{equation*}
So we have 
	\begin{align*}
		\left\|P_m\mathbf{v}\right\|_{2}\leq C \left\|\A P_m \mathbf{v}\right\|\leq C \left\|\A  \mathbf{v}\right\|\leq C\left\|\mathbf{v}\right\|_{2},
	\end{align*}
which yields the desired conclusion. 
\end{proof}


To discuss the boundedness of the solution  $\mathbf{v}_m$ to  the system \cref{appcoe3}, we decompose \( \mathbf{v}_m \) into two components, i.e.
 \begin{equation}\label{decom-vm}
  \mathbf{v}_m = \mathbf{v}^{(1)}_m + \mathbf{v}^{(2)}_m, 
  \end{equation}
 where $  \mathbf{v}^{(1)}_m, \mathbf{v}^{(2)}_m\in \mathscr{P}_m$ satisfy 
\begin{align}
	&\left\{
	\begin{aligned}
		&\frac{\partial \mathbf{v}^{(1)}_m}{\partial t} + \A \mathbf{v}^{(1)}_m = P_m \mathbf{F}, \\
		&\mathbf{v}^{(1)}_m(0) = P_m \mathbf{v}_0,
	\end{aligned}
	\right.
	\label{vm1} \\
	&\left\{
	\begin{aligned}
		&\frac{\partial \mathbf{v}^{(2)}_m}{\partial t} + \A \mathbf{v}^{(2)}_m  =  \int_{0}^{t} \beta(t-s) P_m \B \mathbf{v} _m(s) \, ds, \\
		&\mathbf{v}^{(2)}_m(0) = 0,
	\end{aligned}
	\right.
	\label{vm2}
\end{align}
respectively.
According to the semigroup theory \cite{pazy2012semigroups} and the Duhamel's principle, we have 
\begin{align}
	&\mathbf{v}^{(1)}_m = S(t)P_m\mathbf{v}_0 +  \int_{0}^{t}S(t-s)P_m\mathbf{F}(s)\,ds, \label{vm1solution} \\
	&\mathbf{v}^{(2)}_{m} = \int_{0}^{t}S(t-s)\int_{0}^{s}\beta(s-r)P_m\B\mathbf{v}_m(r)\,dr \,ds. \label{vm2solution}
\end{align}

%

For $\mathbf{v}^{(1)}_m$ we have the following boundedness result:
\begin{lemma}\label{finitevm1regu}
	Under \cref{regularitycond0}  there holds
	\begin{equation}\label{regularityVm1}
		\begin{aligned}
			&\left\| \mathbf{v}^{(1)}_m \right\|_{L^\infty( {\bf H}^2)} + \left\| \frac{\partial \mathbf{v}^{(1)}_m}{\partial t} \right\|_{L^\infty({\bf L}^2)} + \left\| \frac{\partial \mathbf{v}^{(1)}_m}{\partial t} \right\|_{L^1({\bf H}^2)} + \left\| \frac{\partial^2\mathbf{v}^{(1)}_m}{\partial t^2} \right\|_{L^1({\bf L}^2)} \\
			  \leq &C \left( \left\| \mathbf{v}_0\right \|_2 + \left\|\left(\A\mathbf{v}_0+\mathbf{F}(0)\right)\right \|_1+ \left\| \mathbf{F} \right \|_{L^\infty({\bf L}^2)}  + \left\| \frac{\partial \mathbf{F}}{\partial t} \right \|_{L^\infty({\bf L}^2)} + \left\|\frac{\partial^2 \mathbf{F}}{\partial t^2} \right \|_{L^1\left({\bf L}^2\right)} \right).
		\end{aligned}
	\end{equation}
\end{lemma}
\begin{proof}
	We divide the proof  into four steps.
	\begin{steps}
%

\item Let us   estimate  $\left\| \mathbf{v}^{(1)}_m \right\|_{L^\infty( {\bf H}^2)} $.  From  \cref{vm1solution} we obtain
		\begin{equation}\label{vm1_2}
		\begin{aligned}
			&\left\|\mathbf{v}^{(1)}_m\right \|_2\leq \left\|S(t)P_m\mathbf{v}_0\right \|_2+\left\|\int_{0}^{t}S(t-s)P_m\mathbf{F}(s) \, ds\right \|_2\\
	\leq &\left\|S(t)P_m\mathbf{v}_0\right \|_2+\int_{0}^{\frac{t}{2}}\left\|S(t-s)P_m\mathbf{F}(s)\right \|_2 \, ds+\left\|\int_{\frac{t}{2}}^{t}S(t-s)P_m\mathbf{F}(s) \, ds\right \|_2.	
\end{aligned}			
		\end{equation}
	
		 Using  \cref{L2re}, the first inequality in \cref{sgregularity} with $k=0$ and \cref{Pm-v0} we get
		\begin{equation}\label{vm1_2part1}
			\left\|S(t)P_m\mathbf{v}_0\right \|_2\leq C\left\|P_m\mathbf{v}_0\right \|_2\leq C\left\|\mathbf{v}_0\right \|_2.
		\end{equation}
	By the second inequality in \cref{sgregularity} with $k=0$ and $j=1$	 we have
			\begin{equation}\label{vm1_2part3}
			\begin{aligned}
				\int_{0}^{\frac{t}{2}}\left\|S(t-s)P_m\mathbf{F}(s)\right \|_2 \, ds\leq C\int_{0}^{\frac{t}{2}}(t-s)^{-1}\left\|\mathbf{F}(s)\right \| \, ds\leq C\mathop{\text{max}}\limits_{0\leq s\leq t}\left\|\mathbf{F}(s)\right \|.
			\end{aligned}
		\end{equation} 

%

		Applying the regularity inequality \cref{regu-ineq}, the  analytic semigroup property  
		\begin{equation}\label{semigroup1}
		\A S(t-s)=\frac{\partial S(t-s)}{\partial s},
		\end{equation}
		  integration by parts,    the second inequality in \cref{sgregularity} with $k=j=0$ and \cref{Pm-v0} we obtain
		\begin{equation}\label{vm1_2part4}
			\begin{aligned}
				&C\left\|\int_{\frac{t}{2}}^{t}S(t-s)P_m\mathbf{F}(s) \, ds\right \|_2\\
				\leq & \left\|\A\int_{\frac{t}{2}}^{t}S(t-s)P_m\mathbf{F}(s) \, ds\right \|=\left\|\int_{\frac{t}{2}}^{t}\frac{\partial S(t-s)}{\partial s}P_m\mathbf{F}(s) \, ds\right \|\\
			=& \left\|S(0)P_m\mathbf{F}(t)  -S(\frac{t}{2})P_m \mathbf{F}(\frac{t}{2}) -\int_{\frac{t}{2}}^{t} S(t-s)P_m\frac{\partial \mathbf{F}}{\partial t}(s) \, ds\right \|\\
				\leq& \left\|S(0)P_m\mathbf{F}(t)\right \|+\left\|S(\frac{t}{2})P_m\mathbf{F}(\frac{t}{2})\right \|+\int_{\frac{t}{2}}^{t}\left\|S(t-s)P_m\frac{\partial \mathbf{F}}{\partial t}(s)\right \| \, ds\\
				\leq &C\left(\left\|\mathbf{F}(t) \right \|+\left\|\mathbf{F}(\frac{t}{2})\right \|+\mathop{\text{max}}\limits_{0\leq s\leq t}\left\|\frac{\partial \mathbf{F}}{\partial t}(s)\right \|\right).
			\end{aligned}
		\end{equation}

		Combining \cref{vm1_2part1,vm1_2part3,vm1_2part4}, we get
		\begin{align*}
			\left\|\mathbf{v}^{(1)}_m\right \|_2\leq C\left(\left\|\mathbf{v}_0\right \|_2+\mathop{\text{max}}\limits_{0\leq s\leq t}\left\|\mathbf{F}(s)\right \|+\mathop{\text{max}}\limits_{0\leq s\leq t}\left\|\frac{\partial \mathbf{F}}{\partial t}(s)\right \|\right) ,
		\end{align*}
which indicates 
		\begin{equation}\label{v-1-est}
			\left\|\mathbf{v}^{(1)}_m\right \|_{L^\infty\left({\bf H}^2\right)}\leq C\left(\left\|\mathbf{v}_0\right \|_2+\left\|\mathbf{F} \right \|_{L^\infty\left({\bf L}^2\right)}+\left\|\frac{\partial \mathbf{F}}{\partial t} \right \|_{L^\infty\left({\bf L}^2\right)}\right).
		\end{equation}

		\item Let us estimate $ \left\| \frac{\partial \mathbf{v}^{(1)}_m}{\partial t} \right\|_{L^\infty({\bf L}^2)}  $ and $ \left\| \frac{\partial \mathbf{v}^{(1)}_m}{\partial t} \right\|_{L^1({\bf H}^2)}$. Firstly, from \eqref{vm1},  \eqref{H2-A-equiv} and \cref{Pm-v0} 
we have
		\begin{align*}
			\left\|\frac{\partial \mathbf{v}^{(1)}_m}{\partial t}\right\|&\leq \left\|\A\mathbf{v}^{(1)}_m\right\|+\left\|P_m\mathbf{F}\right\| \leq C\left\|\mathbf{v}^{(1)}_m\right\|_2+\left\|\mathbf{F}\right\|,						
		\end{align*}
		
which plus  \cref{v-1-est} further gives
		\begin{equation}\label{v1m_L-inf-L2}
			\left\|\frac{\partial \mathbf{v}^{(1)}_m}{\partial t}\right \|_{L^\infty\left({\bf L}^2\right)}\leq 
			C\left(\left\|\mathbf{v}_0\right \|_2+\left\|\mathbf{F} \right \|_{L^\infty\left({\bf L}^2\right)}+\left\|\frac{\partial \mathbf{F}}{\partial t} \right \|_{L^\infty\left({\bf L}^2\right)}\right).
		\end{equation}
	Secondly, acting $\A$ on \cref{vm1solution} and using \cref{semigroup1}, we get
		\begin{equation*}\label{Lvm1solution}
			\begin{aligned}
				\A\mathbf{v}^{(1)}_m&=\A S(t)P_m\mathbf{v}_0+\A\int_{0}^{t}S(t-s)P_m\mathbf{F}(s) \, ds \\
				&=\A S(t)P_m\mathbf{v}_0-\int_{0}^{t}\frac{\partial S(t-s)}{\partial s}P_m\mathbf{F}(s) \, ds \\
				&=\A S(t)P_m\mathbf{v}_0+S(t)P_m\mathbf{F}(0)-P_m\mathbf{F}(t)+\int_{0}^{t}S(t-s)P_m\frac{\partial \mathbf{F}}{\partial t}(s) \, ds\\
				&=P_mS(t)\left(\A\mathbf{v}_0+\mathbf{F}(0)\right)-P_m\mathbf{F}(t)+\int_{0}^{t}S(t-s)P_m\frac{\partial \mathbf{F}}{\partial t}(s) \, ds.
			\end{aligned}
		\end{equation*}
		Differentiating this equation with respect to $t$, we obtain
		\begin{equation}\nonumber
			\begin{aligned}
				&\A\frac{\partial \mathbf{v}^{(1)}_m}{\partial t}=P_m\frac{\partial S(t)}{\partial t}\left(\A\mathbf{v}_0+\mathbf{F}(0)\right)-P_m\frac{\partial \mathbf{F}}{\partial t}(t)
				+\int_{0}^{t}\frac{\partial S(t-s)}{\partial t}P_m\frac{\partial \mathbf{F}}{\partial t}(s) \, ds\\
				&=P_m\frac{\partial S(t)}{\partial t}\left(\A\mathbf{v}_0+\mathbf{F}(0)\right)-S(t)P_m\frac{\partial \mathbf{F}}{\partial t}(0)-\int_{0}^{t}S(t-s)P_m\frac{\partial^2 \mathbf{F}}{\partial t^2}(s) \, ds.
			\end{aligned}
		\end{equation}
Thus, in light of   the regularity inequality \cref{regu-ineq},  the second and third inequality in \cref{sgregularity} with $k=j=0$ and \cref{Pm-v0} we have 	
%
		\begin{align}\label{3.2100}
			\left\|\frac{\partial \mathbf{v}^{(1)}_m}{\partial t}\right \|_2&\leq C\left\|\A\frac{\partial \mathbf{v}^{(1)}_m}{\partial t}\right \|\\
			&\leq C\left(t^{-\frac{1}{2}}\left\|\left(\A\mathbf{v}_0+\mathbf{F}(0)\right)\right \|_1 
			+\left\|\frac{\partial \mathbf{F}}{\partial t}(0)\right \|+\int_{0}^{t}\left\|\frac{\partial^2 \mathbf{F}}{\partial t^2}(s)\right \|\, ds\right).\nonumber
		\end{align}
	Furthermore,	applying  \cref{regularitycond0} and integrating the inequality over $[0,T]$, we obtain 
		\begin{equation}\label{v1m_L1L2}
			\left\|\frac{\partial \mathbf{v}^{(1)}_m}{\partial t}\right \|_{L^1\left({\bf H}^2\right)}\leq C\left(\left\|\left(\A\mathbf{v}_0+\mathbf{F}(0)\right)\right \|_1+\left\| \frac{\partial \mathbf{F}}{\partial t} \right \|_{L^\infty({\bf L}^2)}+\left\|\frac{\partial^2 \mathbf{F}}{\partial t^2} \right \|_{L^1\left({\bf L}^2\right)}\right).
		\end{equation}
		
		\item Let us estimate  $ \left\|\frac{\partial^2\mathbf{v}^{(2)}_m}{\partial t^2}\right \|_{L^1({\bf L}^2)}$. Differentiating \cref{vm1} with respect to $t$, we arrive at
		
		\begin{equation}\label{v1mtt}
			\frac{\partial^2\mathbf{v}^{(1)}_m}{\partial t^2}+\A\frac{\partial \mathbf{v}^{(1)}_m}{\partial t}=P_m\frac{\partial \mathbf{F}}{\partial t},
		\end{equation}
		which plus \eqref{3.2100} gives
		\begin{equation*}
			\begin{aligned}
				\left\|\frac{\partial^2\mathbf{v}^{(1)}_m}{\partial t^2}\right \|&\leq\left\|\A\frac{\partial \mathbf{v}^{(1)}_m}{\partial t}\right \|+\left\|P_m\frac{\partial \mathbf{F}}{\partial t}\right \|\\
				&\leq C\left(t^{-\frac{1}{2}}\left\|\left(\A\mathbf{v}_0+\mathbf{F}(0)\right)\right \|_1+\left\|\frac{\partial \mathbf{F}}{\partial t} (0)\right \|+\int_{0}^{t}\left\|\frac{\partial^2 \mathbf{F}}{\partial t^2}(s)\right \|\, ds+\left\|\frac{\partial \mathbf{F}}{\partial t} \right \|\right).
			\end{aligned}
		\end{equation*}
		Consequently, 		integrating the above inequality from $0$ to $T$ yields
		\begin{equation}\label{V1m_tt}
			\left\|\frac{\partial^2\mathbf{v}^{(1)}_m}{\partial t^2}\right \|_{L^1\left({\bf L}^2\right)}\leq C\left(\left\|\left(\A\mathbf{v}_0+\mathbf{F}(0)\right)\right \|_1+\left\|\frac{\partial \mathbf{F}}{\partial t} \right \|_{L^\infty\left({\bf L}^2\right)}+\left\|\frac{\partial^2 \mathbf{F}}{\partial t^2} \right \|_{L^1\left({\bf L}^2\right)}\right).
		\end{equation}
		
\item		Finally, combining  \cref{v-1-est,v1m_L-inf-L2,v1m_L1L2,V1m_tt} implies the desired estimate \eqref{regularityVm1}.
	\end{steps}
 \end{proof}

For $\mathbf{v}^{(2)}_m$ we have the following boundedness result:

\begin{lemma}\label{finitevm2regu}
	Under \cref{regularitycond0}  there holds
	\begin{equation}\label{regularityVm2}
		\begin{aligned}
			&\left\|\mathbf{v}^{(2)}_m\right \|_{L^\infty( {\bf H}^2)} + \left\|\frac{\partial \mathbf{v}^{(2)}_m}{\partial t}\right \|_{L^\infty( {\bf L}^2)} + \left\|\frac{\partial \mathbf{v}^{(2)}_m}{\partial t}\right \|_{L^1({\bf H}^2)} + \left\|\frac{\partial^2\mathbf{v}^{(2)}_m}{\partial t^2}\right \|_{L^1({\bf L}^2)} \\
			\leq &C \left( \left\| \mathbf{v}_0\right \|_2 +  \left\|\A\mathbf{v}_0+\mathbf{F}(0) \right \|_1+ \left\| \mathbf{F} \right \|_{L^\infty({\bf L}^2)} + \left\| \frac{\partial \mathbf{F}}{\partial t} \right \|_{L^\infty({\bf L}^2)} + \left\|\frac{\partial^2\mathbf{F}}{\partial t^2} \right \|_{L^1\left({\bf L}^2\right)} \right).
		\end{aligned}
	\end{equation}
\end{lemma}

\begin{proof}
	The proof is divided into four steps.
\begin{steps}
	\item In view of \cref{vm2solution} and \cref{decom-vm}, we decompose $\mathbf{v}^{(2)}_m$ as
		\begin{align*}
		\mathbf{v}^{(2)}_m=Q_1+Q_2,
	\end{align*}
	where
	\begin{align*}
		Q_i=\int_{0}^{t}S(t-s)\int_{0}^{s}\beta(s-r)P_m\B\mathbf{v}^{(i)}_{m}(r) \, dr \, ds, \qquad i=1,\ 2.
	\end{align*}
	For $Q_1$, 
by	\cref{regu-ineq},   \cref{MLbound}, \cref{ML-t-bound}, \cref{fuzhuyinli3,v-1-est} we obtain
	\begin{equation*}
		\begin{aligned}\label{v2m_1}
			\left\|Q_1\right\|_2&=\left\|\int_{0}^{t}S(t-s)\int_{0}^{s}\beta(s-r)P_m\B\mathbf{v}^{(1)}_m(r) \,dr\,ds\right \|_2 \\
			&\leq C\left\|\int_{0}^{t}\A S(t-s)\int_{0}^{s}\beta(s-r)P_m\B\mathbf{v}^{(1)}_m(r) \,dr\,ds\right \| \\
			&= C\left\|\int_{0}^{t}\left(\A\int_{0}^{t-r} S(t-r-s)\beta(s)\,ds\right)P_m\B\mathbf{v}^{(1)}_m(r) \,dr\right \|\\
			&\leq C \int_{0}^{t}\left\|\left(\A\int_{0}^{t-r} S(t-r-s)\beta(s)\,ds\right)P_m\B\mathbf{v}^{(1)}_m(r) \right\|\,dr \\
			&\leq C\int_{0}^{t}(t-r)^{-\alpha}\left\|P_m\B\mathbf{v}^{(1)}_m(r)\right \| \, dr\\
			&\leq C\int_{0}^{t}(t-r)^{-\alpha}\left\|\mathbf{v}^{(1)}_m(r)\right \|_2 \, dr\\
			&\leq C\left(\left\|\mathbf{v}_0\right\|_2+\mathop{\text{max}}\limits_{0\leq s\leq t}\left\|\mathbf{F}(s)\right \|+\mathop{\text{max}}\limits_{0\leq s\leq t}\left\|\frac{\partial \mathbf{F}}{\partial t}(s)\right \|\right)t^{1-\alpha}.
		\end{aligned}
	\end{equation*}
	Similarly, for $Q_2$ we have 
	\begin{equation*}\label{v2m_2}
		\left\|Q_2\right\|_2\leq C\int_{0}^{t}(t-s)^{-\alpha}\left\|\mathbf{v}^{(2)}_m(s)\right \|_2 \, ds.
	\end{equation*}
	Combining the above two estimates, 
	we arrive at
	\begin{align*}
		\left\|\mathbf{v}^{(2)}_m\right \|_2 &\leq \left\|Q_1\right\|_2+\left\|Q_2\right\|_2\\
		&\leq C\Biggl\{ 
		\Biggl( \left\|\mathbf{v}_0\right \|_2 + \max_{0 \leq s \leq t} \left\|\mathbf{F}(s)\right \| + \max_{0 \leq s \leq t} \left\|\frac{\partial \mathbf{F}}{\partial t}(s)\right \| \Biggr) t^{1-\alpha} \\
		& \qquad \qquad\qquad \qquad + \int_{0}^{t} (t-s)^{-\alpha} \left\|\mathbf{v}^{(2)}_m(s)\right \|_2 \, ds 
		\Biggr\}.
	\end{align*}
	Thus, applying the Gronwall's inequality (\cref{fuzhuyinli2}), we obtain
	\begin{align*}
		\left\|\mathbf{v}^{(2)}_m\right \|_2\leq C\left(\left\|\mathbf{v}_0\right \|_2+\mathop{\text{max}}\limits_{0\leq s\leq t}\left\|\mathbf{F}(s)\right \|+\mathop{\text{max}}\limits_{0\leq s\leq t}\left\|\frac{\partial \mathbf{F}}{\partial t}(s)\right \|\right)t^{1-\alpha},
	\end{align*}
	which indicates
	\begin{equation}\label{v-2-est}
		\left\|\mathbf{v}^{(2)}_m\right \|_{L^\infty\left({\bf H}^2\right)}\leq C\left(\left\|\mathbf{v}_0\right \|_2+\left\|\mathbf{F} \right \|_{L^\infty\left({\bf L}^2\right)}+\left\|\frac{\partial \mathbf{F}}{\partial t} \right \|_{L^\infty\left({\bf L}^2\right)}\right).
	\end{equation}
	
	\item Let us estimate $\left\|\frac{\partial \mathbf{v}^{(2)}_m}{\partial t}\right \|_{L^\infty\left({\bf L}^2\right)}$ and $\left\|\frac{\partial \mathbf{v}^{(2)}_m}{\partial t}\right \|_{L^1\left({\bf H}^2\right)}$. Firstly, from \cref{vm2,H2-A-equiv,MLbound,L2re} we have
	\begin{equation}
		\begin{aligned}
			\left\|\frac{\partial \mathbf{v}^{(2)}_m}{\partial t}\right \|&\leq \left\|\A\mathbf{v}^{(2)}_m\right \|+\left\|\int_{0}^{t}\beta(t-s)P_m\B\mathbf{v}_m(s) \, ds\right \| \\
			&\leq\left\|\mathbf{v}^{(2)}_m\right \|_2+\int_{0}^{t}(t-s)^{-\alpha}\left\|\mathbf{v}_m(s)\right \|_2 \, ds,
		\end{aligned}
	\end{equation}
	which plus \cref{v-1-est,v-2-est} further gives
	\begin{equation}\label{vm2-t-L2}
		\left\|\frac{\partial \mathbf{v}^{(2)}_m}{\partial t}\right \|_{L^\infty\left({\bf L}^2\right)}\leq C\left(\left\|\mathbf{v}_0\right \|_2+\left\|\mathbf{F} \right \|_{L^\infty\left({\bf L}^2\right)}+\left\|\frac{\partial \mathbf{F}}{\partial t} \right \|_{L^\infty\left({\bf L}^2\right)}\right).
	\end{equation}
	Secondly, differentiating equation \cref{vm2solution} with respect to $t$ and 
	applying 
	 integration by parts,  we obtain 
	\begin{equation}\label{vm2_t}
		\begin{aligned}
			\frac{\partial \mathbf{v}^{(2)}_m}{\partial t}&=\int_{0}^{t}\beta(t-r)P_m\B\mathbf{v}_m(r) \,dr+\int_{0}^{t}\frac{\partial S(t-s)}{\partial t}\int_{0}^{s}\beta(s-r)P_m\B\mathbf{v}_m(r) \,dr \,ds\\
			&=\int_{0}^{t}\beta(t-r)P_m\B\mathbf{v}_m(r) \,dr-\int_{0}^{t}\frac{\partial S(t-s)}{\partial s}\int_{0}^{s}\beta(s-r)P_m\B\mathbf{v}_m(r) \,dr \,ds\\
			&=\int_{0}^{t}S(t-s)\beta(s)P_m\B\mathbf{v}_0 \, ds+\int_{0}^{t}S(t-s)\int_{0}^{s}\beta(r)P_m\B\frac{\partial\mathbf{v}_m}{\partial s}(s-r) \,dr \,ds.
		\end{aligned}
	\end{equation}
In light of 
	\cref{MLbound} and the first inequality in \cref{sgregularity} with $k=0$ we get
	\begin{equation}\label{331}
		\left\|\int_{0}^{t}S(t-s)\beta(s)P_m\B\mathbf{v}_0 \, ds\right \|_2\leq C t^{-\alpha}\left\|\mathbf{v}_0\right \|_2.
	\end{equation}
	By \cref{regu-ineq,fuzhuyinli3} we find
	\begin{equation}\label{v-2-t-H2}
		\begin{aligned}
			&\left\|\int_{0}^{t}S(t-s)\int_{0}^{s}\beta(r)P_m\B\frac{\partial\mathbf{v}_m}{\partial s}(s-r) \,dr \,ds\right \|_2\\
			&\leq C\left\|\int_{0}^{t}\A S(t-s)\int_{0}^{s}\beta(r)P_m\B\frac{\partial\mathbf{v}_m}{\partial s}(s-r) \,dr \,ds\right \|\\
			&= C\left\|\int_{0}^{t}\left(\A \int_{0}^{t-r}S(t-r-s)\beta(s)\,ds\right)P_m\B\frac{\partial\mathbf{v}_m}{\partial s}(r) \,dr \right \|\\
			&\leq C\int_{0}^{t}(t-r)^{-\alpha}\left\|P_m\B\frac{\partial\mathbf{v}_m}{\partial s}(r)\right \| \, dr\\
			&\leq C\int_{0}^{t}(t-r)^{-\alpha}\left\|\frac{\partial\mathbf{v}_m}{\partial t}(r)\right \|_2 \, dr\\
			&= C\left(\int_{0}^{t}(t-r)^{-\alpha}\left\|\frac{\partial \mathbf{v}^{(1)}_m}{\partial t}(r)\right \|_2 \, dr+\int_{0}^{t}(t-r)^{-\alpha}\left\|\frac{\partial \mathbf{v}^{(2)}_m}{\partial t}(r)\right \|_2\, dr\right).
		\end{aligned}
	\end{equation}
	Combining \cref{v1m_L1L2,vm2_t,331,v-2-t-H2}, we obtain
	\begin{equation*}
		\begin{aligned}
			\left\|\frac{\partial\mathbf{v}^{(2)}_m}{\partial t}\right\|_2
			\leq C \biggl( 
			t^{-\alpha}\left\|\mathbf{v}_0\right \|_2+\left\|\left(\A\mathbf{v}_0+\mathbf{F}(0)\right)\right \|_1&+\left\| \frac{\partial \mathbf{F}}{\partial t} \right \|_{L^\infty({\bf L}^2)}+\left\|\frac{\partial^2 \mathbf{F}}{\partial t^2} \right \|_{L^1\left({\bf L}^2\right)} \\
			& + \int_{0}^{t}(t-s)^{-\alpha}\left\|\frac{\partial \mathbf{v}^{(2)}_m}{\partial t}(s)\right\|_2 \, ds  \biggr),
		\end{aligned}
	\end{equation*}
which, 	
together with  \cref{fuzhuyinli2}, yields
	\begin{equation}\label{vm2-t-H2}
		\left\|\frac{\partial \mathbf{v}^{(2)}_m}{\partial t}\right \|_{L^1\left({\bf H}^2\right)}\leq C\left(\left\|\mathbf{v}_0\right \|_2+\left\|\left(\A\mathbf{v}_0+\mathbf{F}(0)\right)\right \|_1+\left\| \frac{\partial \mathbf{F}}{\partial t} \right \|_{L^\infty({\bf L}^2)}+\left\|\frac{\partial^2 \mathbf{F}}{\partial t^2} \right \|_{L^1\left({\bf L}^2\right)}\right).
	\end{equation}

	\item Let us estimate $\|\frac{\partial^2\mathbf{v}^{(2)}_m}{\partial t^2}\|_{L^1({\bf L}^2)}$. Differentiating equation \cref{vm2} with respect to $t$, we have
	\begin{equation}\label{vm2tt}
		\frac{\partial^2\mathbf{v}^{(2)}_m}{\partial t^2}+\A\frac{\partial \mathbf{v}^{(2)}_m}{\partial t}=\beta(t)P_m\B\mathbf{v}_m(0)+\int_{0}^{t}\beta(s)P_m\B\frac{\partial\mathbf{v}_m}{\partial t}(t-s) \, ds,
	\end{equation}
	which plus \cref{H2-A-equiv} gives
	\begin{equation}\nonumber
		\begin{aligned}
			\left\|\frac{\partial^2\mathbf{v}^{(2)}_m}{\partial t^2}\right \|&\leq\left\|\A\frac{\partial \mathbf{v}^{(2)}_m}{\partial t}\right\|+\left\|\beta(t)P_m\B\mathbf{v}_m(0)\right\|+\left\|\int_{0}^{t}\beta(s)P_m\B\frac{\partial\mathbf{v}_m}{\partial t}(t-s) \, ds\right\| \\
			&\leq  C\left(\left\|\frac{\partial \mathbf{v}^{(2)}_m}{\partial t}\right \|_2+t^{-\alpha}\left\|\mathbf{v}_0\right \|_2+\int_{0}^{t}(t-s)^{-\alpha}\left\|\frac{\partial\mathbf{v}_m}{\partial t}(s)\right\|_2 \, ds
			\right).
		\end{aligned}
	\end{equation}
	Hence, integrating the above inequality from $0$ to $T$ indicates
	\begin{equation}\label{vm2-tt-L2}
		\left\|\frac{\partial^2\mathbf{v}^{(2)}_m}{\partial t^2}\right \|_{L^1\left({\bf L}^2\right)}\leq C\left(\left\|\mathbf{v}_0\right \|_2+\left\|\left(\A\mathbf{v}_0+\mathbf{F}(0)\right)\right \|_1+\left\| \frac{\partial \mathbf{F}}{\partial t} \right \|_{L^\infty({\bf L}^2)}+\left\|\frac{\partial^2 \mathbf{F}}{\partial t^2} \right \|_{L^1\left({\bf L}^2\right)}\right).
	\end{equation}

	\item Finally, combining \cref{v-2-est,vm2-t-L2,vm2-t-H2,vm2-tt-L2} implies the desired estimate \cref{regularityVm2}.
\end{steps}

\end{proof}

In light of \cref{finitevm1regu,finitevm2regu}, we immediately arrive at the
following  stability result:
\begin{theorem}\label{finitedimenregu}
	Under \cref{regularitycond0} there holds
	\begin{equation}\label{energyestimates}
		\begin{aligned}
			&\left\|\mathbf{v}_m\right\|_{L^\infty({\bf H}^2)} + \left\|\frac{\partial\mathbf{v}_m}{\partial t}\right\|_{L^\infty({\bf L}^2)} + \left\|\frac{\partial\mathbf{v}_m}{\partial t}\right\|_{L^1({\bf H}^2)} + \left\|\frac{\partial^2\mathbf{v}_m}{\partial t^2}\right\|_{L^1({\bf L}^2)} \\
			\leq &C \left( \left\| \mathbf{v}_0\right \|_2  + \left\|\left(\A\mathbf{v}_0+\mathbf{F}(0)\right)\right \|_1 + \left\| \mathbf{F} \right \|_{L^\infty({\bf L}^2)} + \left\| \frac{\partial \mathbf{F}}{\partial t} \right \|_{L^\infty({\bf L}^2)} + \left\|\frac{\partial^2 \mathbf{F}}{\partial t^2}\right\|_{L^1({\bf L}^2)} \right).
		\end{aligned}
	\end{equation}
\end{theorem}

\subsection{Existence, uniqueness and regularity of weak solution}

Based on \cref{finitedimenregu}, we   follow the standard line of the Galerkin method  (cf. \cite{Evans2010}) to show 
 the wellposedness of     the weak problem  \cref{weak problem}.
 
\begin{theorem}\label{cunzaiweiyixing}
	Under \cref{regularitycond0}  the   problem \cref{weak problem} admits  a unique solution $\mathbf{v}(t) $, and there holds 
	\begin{equation}\label{weaksolure} 
		\begin{aligned}
			&\left\|\mathbf{v}\right\|_{L^\infty({\bf H}^2)} + \left\|\frac{\partial\mathbf{v}}{\partial t}\right\|_{L^\infty({\bf L}^2)} + \left\|\frac{\partial\mathbf{v}}{\partial t}\right\|_{L^1({\bf H}^2)} + \left\|\frac{\partial^2\mathbf{v}}{\partial t^2}\right\|_{L^1({\bf L}^2)} \\
			\leq &C \left( \left\| \mathbf{v}_0\right \|_2  + \left\|\left(\A\mathbf{v}_0+\mathbf{F}(0)\right)\right \|_1 + \left\| \mathbf{F} \right \|_{L^\infty({\bf L}^2)} + \left\| \frac{\partial \mathbf{F}}{\partial t} \right \|_{L^\infty({\bf L}^2)} + \left\|\frac{\partial^2 \mathbf{F}}{\partial t^2}\right\|_{L^1({\bf L}^2)} \right).
		\end{aligned}
	\end{equation}
%
\end{theorem}
\begin{proof}
	Firstly, according to   \cref{energyestimates} we know that
	\begin{align*}
		&\left\{\mathbf{v}_m\right\}_{m=1}^{\infty}\text{ is bounded in }L^\infty\left({\bf H}^2(\Omega)\cap {\bf H}^1_0(\Omega)\right),\text{ and }\\
		&\left\{\frac{\partial\mathbf{v}_m}{\partial t}\right\}_{m=1}^{\infty}\text{ is bounded in }L^\infty\left({\bf L}^2(\Omega)\right).
	\end{align*} 
Thus, there exists a subsequence of $  \left\{\mathbf{v}_m\right\}_{m=1}^{\infty}$, still denoted by $ \left\{\mathbf{v}_m\right\}_{m=1}^{\infty}$ for simplicity,  and a function $\mathbf{v}\in L^\infty\left([0,T];{\bf H}^2(\Omega)\cap {\bf H}^1_0(\Omega)\right)$ with $\frac{\partial\mathbf{v}}{\partial t}\in L^\infty\left([0,T];{\bf L}^2(\Omega)\right)$, such that
	\begin{equation}\label{weak convergence}
		\left\{
		\begin{array}{ll}
			\mathbf{v}_{m_l}\stackrel{*}{\rightharpoonup} \mathbf{v} \qquad &\text{in }  L^\infty\left([0,T];{\bf H}^2(\Omega)\cap {\bf H}^1_0(\Omega)\right), \\
			\frac{\partial\mathbf{v}_{m_l}}{\partial t}\stackrel{*}{\rightharpoonup} \frac{\partial\mathbf{v}}{\partial t} \qquad &\text{in }  L^\infty\left([0,T];{\bf L}^2(\Omega)\right),
		\end{array}
		\right.
	\end{equation}
where $\stackrel{*}{\rightharpoonup}$ denotes the weak- * convergence.

Next, we demonstrate that $\mathbf{v}$ is a solution of \cref{weak problem}. Fix an integer $N$, and select a function $\mathbf{w}\in C^1\left({\bf H}_0^1(\Omega)\right)$ of the form
\begin{equation}\label{hanshu1}
	\mathbf{w}(t)=\sum_{k=1}^{N}\tilde d_k(t)\bm{\varphi}_k,
\end{equation}
where $\tilde d_k(t)\in C^1([0,T])$ are given  functions. Then from   \cref{appcoe3}
  we have
\begin{equation}\label{youxianxingshi}
	\int_{0}^{T}\langle\frac{\partial\mathbf{v}_m}{\partial t},\mathbf{w}\rangle+a(\mathbf{v}_m,\mathbf{w}) \, dt-\int_{0}^{T}\int_{0}^{t}\beta(t-s)b(\mathbf{v}_m(s),\mathbf{w})\, ds \,dt=\int_{0}^{T}\langle \mathbf{F},\mathbf{w}\rangle \, dt.
\end{equation}
Take the limit as $m\rightarrow\infty$ to obtain
\begin{equation}\label{limitxingshi}
	\int_{0}^{T}\langle\frac{\partial\mathbf{v}}{\partial t},\mathbf{w}\rangle+a(\mathbf{v},\mathbf{w}) \, dt-\int_{0}^{T}\int_{0}^{t}\beta(t-s)b(\mathbf{v}(s),\mathbf{w})\, ds \, dt=\int_{0}^{T}\langle \mathbf{F},\mathbf{w}\rangle \, dt.
\end{equation}
Note that the above equality holds for all $\mathbf{w}\in L^1\left([0,T];{\bf H}_0^1(\Omega)\right)$, since functions of the form given in \cref{hanshu1} are dense in $L^1\left([0,T];{\bf H}_0^1(\Omega)\right)$. Hence, we see that $\mathbf{v}$ satisfies the first equation of \cref{weak problem}.

 Furthermore, we need to show that $\mathbf{v}$ satisfies the initial condition. Notice that \cref{limitxingshi,youxianxingshi} hold for any $\mathbf{w}\in C^1\left([0,T];{\bf H}_0^1(\Omega)\right)$ with $\mathbf{w}(T)=0$, then  apply  integration by parts to get
\begin{equation}\label{342}
	\begin{aligned}
		&\int_{0}^{T}\left(-\langle\mathbf{v},\frac{\partial\mathbf{w}}{\partial t}\rangle+a(\mathbf{v},\mathbf{w}) \,\right) dt-\int_{0}^{T}\int_{0}^{t}\beta(t-s)b(\mathbf{v}(s),\mathbf{w})\, ds \,dt\\
		=&\int_{0}^{T}\langle \mathbf{F},\mathbf{w}\rangle \, dt-\langle\mathbf{v}(0),\mathbf{w}(0)\rangle
	\end{aligned}
\end{equation}
and 
\begin{equation}\label{343}
	\begin{aligned}
		&\int_{0}^{T}\left(-\langle\mathbf{v}_{m},\frac{\partial\mathbf{w}}{\partial t}\rangle+a(\mathbf{v}_m,\mathbf{w}) \, \right)dt-\int_{0}^{T}\int_{0}^{t}\beta(t-s)b(\mathbf{v}_m(s),\mathbf{w})\,ds \,dt\\
		=&\int_{0}^{T}\langle \mathbf{F},\mathbf{w}\rangle \, dt-\langle\mathbf{v}_m(0),\mathbf{w}(0)\rangle.
	\end{aligned}
\end{equation}
Letting $m\rightarrow\infty$ in \cref{343} and using \cref{weak convergence},  we get
\begin{align*}
	&\int_{0}^{T}\langle\mathbf{v},\frac{\partial\mathbf{w}}{\partial t}\rangle+a(\mathbf{v},\mathbf{w}) \, dt-\int_{0}^{T}\int_{0}^{t}\beta(t-s)b(\mathbf{v}(s),\mathbf{w})\, ds \,dt\\
	=&\int_{0}^{T}\langle \mathbf{F},\mathbf{w}\rangle \, dt-\langle\mathbf{v}_0,\mathbf{w}(0)\rangle.
\end{align*}
which plus \cref{342} yields $\mathbf{v}(0)=\mathbf{v}_0$.  Therefore, $\mathbf{v}$ is the solution of weak problem \cref{weak problem}.


 By   \cref{finitedimenregu},  taking $ m \to \infty $ in  \cref{energyestimates} gives the   regularity result  
  \cref{weaksolure}. 

The uniqueness of the weak solution follows from \cref{weaksolure} directly, since $ \mathbf{F} =0$ and $ \mathbf{v}_0 = 0 $ imply  $ \mathbf{v} = 0 $. 

%
%
%
\end{proof}

\begin{remark}
	Note that   higher regularity of the solution of the weak problem \cref{weak problem} can be obtained similarly under higher regularity assumptions of the initial data $\mathbf{v}_0$ and the right term $\mathbf{F}$. 
\end{remark}

\section{Semi-discrete finite element method}

In this section, we consider the semi-discrete finite element discretization of the weak problem \cref{weak problem}. 

\subsection{Semi-discrete scheme}  
Let $\mathcal{T}_h=\bigcup\{K\}$ be a shape regular partition of  $\Omega$ consisting of triangles/tetrahedrons or rectangles/cuboids.  For any $K\in \mathcal{T}_h$, let $h_K$ denote its diameter, and set $h:=\mathop{\mathrm{max}}_{K\in\mathcal{T}_h} h_K$. 

Let $V_h\subset {\bf H}_0^1(\Omega)$ be a   linear/bilinear/trilinear finite element space  satisfying 
\begin{equation}\label{finitedimensionas}
	\inf_{\bm{\chi}\in V_h}\left\{\left\|\mathbf{w}-\bm{\chi}\right\|+h\left\|\mathbf{w}-\bm{\chi}\right\|_1\right\}\leq Ch^2\|\mathbf{w}\|_2,\qquad \forall\ \mathbf{w}\in {\bf H}^2(\Omega)\cap {\bf H}_0^1(\Omega).
\end{equation}
Furthermore, the following inverse inequality holds: for any $\mathbf{w}_h\in V_h$, there holds
\begin{equation}\label{inverseassum}
	\|\mathbf{w}_h\|_1\leq Ch^{-1}\|\mathbf{w}_h\|.
\end{equation} 

Then the semi-discrete finite element scheme reads: For $t\in (0,T]$, find $\mathbf{v}_h(t)\in V_h$ such that
\begin{equation}\label{semidiscrete}
	\left\{
	\begin{array}{lll}
		\langle\frac{\partial\mathbf{v}_h}{\partial t},\mathbf{w}_h\rangle+a(\mathbf{v}_h,\mathbf{w}_h)-\displaystyle\int_{0}^{t}\beta(t-s)b(\mathbf{v}_h(s),\mathbf{w}_h) \, ds=\langle \mathbf{F},\mathbf{w}_h\rangle, \ \forall\ \mathbf{w}_h\in V_h, \\ 
		\mathbf{v}_h(0)=\mathbf{v}_{0,h},
	\end{array}
	\right.
\end{equation}
where $\mathbf{v}_{0,h}=\Pi_h\mathbf{v}_0\in V_h$ with $\Pi_h$ being  the Ritz-Volterra projection operator defined in \cref{projections-RV}. Note that there holds (cf. \cref{RVprojection})
\begin{equation}\label{initialassu}
	\left\|\Pi_h\mathbf{v}_0-\mathbf{v}_0\right\|\leq Ch^2\left\|\mathbf{v}_0\right \|_2.
\end{equation} 

 In light of \cref{solutionofintegrodifferential} we easily obtain the following conclusion:
\begin{lemma}
	Under Assumption \ref{regularitycond0} the semi-discrete scheme \cref{semidiscrete} admits a unique solution.
\end{lemma}

\subsection{Ritz-Volterra projection}
To carry out the semi-discrete error estimation,  we introduce the standard Ritz projection $R_h$ and the Ritz-Volterra projection $\Pi_h$ (cf.  \cite{cannon1990priori}) as follows: for   $t\in (0,T]$ and  $\mathbf{w}(t)\in {\bf H}_0^1(\Omega)$, 
\begin{equation}\label{projections-R}
		a(R_h\mathbf{w}-\mathbf{w},\bm{\chi})=0, \qquad\ \forall\ \bm{\chi}\in V_h,
\end{equation}
\begin{equation}\label{projections-RV}
		a((\Pi_h\mathbf{w}-\mathbf{w})(t),\bm{\chi})=\int_0^t\beta(t-s)b((\Pi_h\mathbf{w}-\mathbf{w})(s),\bm{\chi}) \, ds,  \ \forall\ \bm{\chi}\in V_h.
\end{equation}

For the Ritz projection $R_h$, under the assumption \cref{finitedimensionas} there holds (cf. \cite{Ciarlet2002})
\begin{equation}\label{est-Rh}
	\|R_h\mathbf{w}-\mathbf{w}\|+h\|R_h\mathbf{w}-\mathbf{w}\|_1\leq Ch^2\|\mathbf{w}\|_2,\quad \forall\ \mathbf{w}\in {\bf H}^2(\Omega)\cap {\bf H}_0^1(\Omega).
\end{equation}

Based on the estimate \cref{est-Rh},     \cref{RVprojection,RVprojection_t}  below  give error estimates for the Ritz-Volterra projection $\Pi_h$. 

%

\begin{lemma}\label{RVprojection}
	For any $   \mathbf{w}\in C^0 ({\bf H}^2(\Omega)\cap {\bf H}_0^1(\Omega))$ there holds 
	\begin{equation}\label{esti-RV}
		\|\Pi_h\mathbf{w}-\mathbf{w}\|_{L^\infty({\bf L}^2)}+h\|\Pi_h\mathbf{w}-\mathbf{w}\|_{L^\infty({\bf H}^1)}\leq Ch^2\|\mathbf{w}\|_{L^\infty({\bf H}^2)}.
	\end{equation}
\end{lemma}
\begin{proof}
	From \cref{a-property,projections-R,projections-RV,MLbound}  it follows that
\begin{equation}\nonumber
	\begin{aligned}
		C\|\Pi_h\mathbf{w}-R_h\mathbf{w}\|_1^2&\leq a(\Pi_h\mathbf{w}-R_h\mathbf{w},\Pi_h\mathbf{w}-R_h\mathbf{w})=a(\Pi_h\mathbf{w}-\mathbf{w},\Pi_h\mathbf{w}-R_h\mathbf{w}) \\
		&=\int_0^t\beta(t-s)b((\Pi_h\mathbf{w}-\mathbf{w})(s),\Pi_h\mathbf{w}-R_h\mathbf{w}) \, ds\\
		&\leq C\|\Pi_h\mathbf{w}-R_h\mathbf{w}\|_1\cdot\int_{0}^{t}(t-s)^{-\alpha}\|(\Pi_h\mathbf{w}-\mathbf{w})(s)\|_1 \, ds.
	\end{aligned}
\end{equation}
Hence, we deduce
\begin{equation}\nonumber
	\begin{aligned}
		\|\Pi_h\mathbf{w}-\mathbf{w}\|_1&\leq \|\Pi_h\mathbf{w}-R_h\mathbf{w}\|_1+\|R_h\mathbf{w}-\mathbf{w}\|_1\\
		&\leq \|R_h\mathbf{w}-\mathbf{w}\|_1+C\int_{0}^{t}(t-s)^{-\alpha}\|(\Pi_h\mathbf{w}-\mathbf{w})(s)\|_1 \, ds.
	\end{aligned}
\end{equation}
By virtue of \cref{fuzhuyinli2,est-Rh} we further obtain
\begin{equation}\label{est4.8}
	\begin{aligned}
		\|\Pi_h\mathbf{w}-\mathbf{w}\|_1&\leq\|R_h\mathbf{w}-\mathbf{w}\|_1+C\int_{0}^{t}(t-s)^{-\alpha}\|(R_h\mathbf{w}-\mathbf{w})(s)\|_1 \, ds \\
		&\leq C\sup_{0\leq s\leq t}\|(R_h\mathbf{w}-\mathbf{w})(s)\|_1\\
		&\leq Ch\sup_{0\leq s\leq t}\|\mathbf{w}(s)\|_2.
	\end{aligned}
\end{equation}
Now we turn to estimate $\|\Pi_h\mathbf{w}-\mathbf{w}\|$ by a duality argument. 
Let $\bm{\phi}$ be the solution of the   equation
\begin{equation}\label{fuzhufangcheng}
	\left\{
	\begin{aligned}
		\A\bm{\phi}&=\Pi_h\mathbf{w}-\mathbf{w} \qquad \text{in}\ \Omega, \\
		\bm{\phi}&=0 \qquad \text{on}\ \partial\Omega.
	\end{aligned}
	\right.
\end{equation}
From \cref{regu-ineq} we see that 
\begin{equation}\label{regu-dual}
\|\bm{\phi}\|_2\leq C\|\Pi_h\mathbf{w}-\mathbf{w}\|.
\end{equation}

In view of  \cref{fuzhufangcheng,projections-RV} we have 
\begin{equation}\label{fuzhuequality1}
	\begin{aligned}
		\|\Pi_h\mathbf{w}-\mathbf{w}\|^2&=\langle\Pi_h\mathbf{w}-\mathbf{w},\A\bm{\phi}\rangle=a(\Pi_h\mathbf{w}-\mathbf{w},\bm{\phi})\\
		&=a(\Pi_h\mathbf{w}-\mathbf{w},\bm{\phi}-R_h\bm{\phi})+a(\Pi_h\mathbf{w}-\mathbf{w},R_h\bm{\phi})\\
		& =a(\Pi_h\mathbf{w}-\mathbf{w},\bm{\phi}-R_h\bm{\phi})+\int_0^t\beta(t-s)b((\Pi_h\mathbf{w}-\mathbf{w})(s),R_h\bm{\phi}) \, ds \\
		&=a(\Pi_h\mathbf{w}-\mathbf{w},\bm{\phi}-R_h\bm{\phi})+\int_0^t\beta(t-s)b((\Pi_h\mathbf{w}-\mathbf{w})(s),\bm{\phi}) \, ds\\
		&\hskip2cm +\int_0^t\beta(t-s)b((\Pi_h\mathbf{w}-\mathbf{w})(s),R_h\bm{\phi}-\bm{\phi}) \, ds\\
		&=a(\Pi_h\mathbf{w}-\mathbf{w},\bm{\phi}-R_h\bm{\phi})-\int_0^t\beta(t-s)\langle(\Pi_h\mathbf{w}-\mathbf{w})(s),\B\bm{\phi}\rangle \, ds\\
		&\hskip2cm +\int_0^t\beta(t-s)b((\Pi_h\mathbf{w}-\mathbf{w})(s),R_h\bm{\phi}-\bm{\phi}) \, ds,
	\end{aligned}
\end{equation}
which, together with   \cref{est-Rh,est4.8}, gives   
\begin{equation}
	\begin{aligned}
		&\|\Pi_h\mathbf{w}-\mathbf{w}\|^2\\
		& 
		\leq C\left(\|\Pi_h\mathbf{w}-\mathbf{w}\|_1\cdot\|\bm{\phi}-R_h\bm{\phi}\|_1+ \int_{0}^{t}(t-s)^{-\alpha} \|(\Pi_h\mathbf{w}-\mathbf{w})(s)\| \, ds \cdot \|\bm{\phi}\|_2 \right.  \\
		&\qquad   \left.+ \int_{0}^{t}(t-s)^{-\alpha} \|(\Pi_h\mathbf{w}-\mathbf{w})(s)\|_1 \, ds \cdot \|R_h\bm{\phi}-\bm{\phi}\|_1 \right) \\
		&\leq C\left\{h^2\|\bm{\phi}\|_2\sup_{0\leq s\leq t}\|\mathbf{w}(s)\|_2+\int_{0}^{t}(t-s)^{-\alpha}\|(\Pi_h\mathbf{w}-\mathbf{w})(s)\| \, ds\right\}.
	\end{aligned}
\end{equation}

Finally, using \cref{regu-dual,fuzhuyinli2}
 we immediately obtain
\begin{equation*}
	\|\Pi_h\mathbf{w}-\mathbf{w}\|\leq Ch^2\sup_{0\leq s\leq t}\|\mathbf{w}(s)\|_2\leq Ch^2\|\mathbf{w}\|_{L^\infty({\bf H}^2)},
\end{equation*}
which plus \cref{est4.8} leads to the desired estimate \cref{esti-RV}.
\end{proof}

\begin{lemma}\label{RVprojection_t}
	For any $   \mathbf{w}\in C^1 ({\bf H}^2(\Omega)\cap {\bf H}_0^1(\Omega))$ 
	 there holds
	\begin{equation}\label{estRV}
		\begin{aligned}
			\int_{0}^{t}\left(\left\|\frac{\partial(\Pi_h\mathbf{w}-\mathbf{w})}{\partial t}\right\|+h\left\|\frac{\partial(\Pi_h\mathbf{w}-\mathbf{w})}{\partial t}\right\|_1\right) \, ds\leq Ch^2\left(\left\|\mathbf{w}(0)\right \|_2+\int_{0}^{t}\left\|\frac{\partial\mathbf{w}}{\partial t}\right\|_2 \, ds\right).
		\end{aligned}
	\end{equation}
\end{lemma}
\begin{proof}
	Rewrite \cref{projections-RV} as
	\begin{equation}
		a(\Pi_h\mathbf{w}-\mathbf{w},\bm{\chi})=\int_0^t\beta(s)b((\Pi_h\mathbf{w}-\mathbf{w})(t-s),\bm{\chi}) \, ds, \qquad \forall\ \bm{\chi}\in V_h.
	\end{equation}
Taking the derivative with respect to $t$ for the above equation, we obtain
\begin{equation*}
	a(\frac{\partial(\Pi_h\mathbf{w}-\mathbf{w})}{\partial t}(t),\bm{\chi})=\beta(t)b(\Pi_h\mathbf{w}(0)-\mathbf{w}(0),\bm{\chi})+\int_0^t\beta(s)b(\frac{\partial(\Pi_h\mathbf{w}-\mathbf{w})}{\partial t}(t-s),\bm{\chi}) \, ds.
\end{equation*}
Thus, by \cref{a-property,projections-R,MLbound} we have
\begin{equation}\nonumber
	\begin{aligned}
		&C\left\|\frac{\partial(\Pi_h\mathbf{w}-R_h\mathbf{w})}{\partial t}\right\|_1^2\\
		\leq &a(\frac{\partial(\Pi_h\mathbf{w}-R_h\mathbf{w})}{\partial t},\frac{\partial(\Pi_h\mathbf{v}-R_h\mathbf{w})}{\partial t}) 
		=a(\frac{\partial(\Pi_h\mathbf{w}-\mathbf{w})}{\partial t},\frac{\partial(\Pi_h\mathbf{w}-R_h\mathbf{w})}{\partial t}) \\
		=&\beta(t)b(\Pi_h\mathbf{w}(0)-\mathbf{w}(0),\frac{\partial(\Pi_h\mathbf{w}-R_h\mathbf{w})}{\partial t})\\
		&\phantom{\beta(t)C C}+\int_0^t\beta(t-s)b(\frac{\partial(\Pi_h\mathbf{w}-\mathbf{w})}{\partial t}(s),\frac{\partial(\Pi_h\mathbf{w}-R_h\mathbf{w})}{\partial t}) \, ds\\
		\leq &C\left(t^{-\alpha}\|\Pi_h\mathbf{w}(0)-\mathbf{w}(0)\|_1+\int_{0}^{t}(t-s)^{-\alpha}\left\|\frac{\partial(\Pi_h\mathbf{w}-\mathbf{w})}{\partial t}(s)\right\|_1 \, ds\right)\left\|\frac{\partial(\Pi_h\mathbf{w}-R_h\mathbf{w})}{\partial t}\right\|_1\\
		\leq &C\left(t^{-\alpha}h\left\|\mathbf{w}(0)\right \|_2+\int_{0}^{t}(t-s)^{-\alpha}\left\|\frac{\partial(\Pi_h\mathbf{w}-\mathbf{w})}{\partial t}(s)\right\|_1 \, ds\right)\left\|\frac{\partial(\Pi_h\mathbf{w}-R_h\mathbf{w})}{\partial t}\right\|_1.
	\end{aligned}
\end{equation}
As a result, we get
\begin{equation}\nonumber
	\begin{aligned}
		\left\|\frac{\partial(\Pi_h\mathbf{w}-\mathbf{w})}{\partial t}\right\|_1&\leq\left\|\frac{\partial(\Pi_h\mathbf{w}-R_h\mathbf{w})}{\partial t}\right\|_1+\left\|\frac{\partial(R_h\mathbf{w}-\mathbf{w})}{\partial t}\right\|_1\\
		&\leq C\left\{h\left(t^{-\alpha}\left\|\mathbf{w}(0)\right \|_2+\left\|\frac{\partial\mathbf{w}}{\partial t}\right\|_2\right)+\int_{0}^{t}(t-s)^{-\alpha}\left\|\frac{\partial(\Pi_h\mathbf{w}-\mathbf{w})}{\partial t}(s)\right\|_1 \, ds\right\},
	\end{aligned}
\end{equation}
which plus \cref{fuzhuyinli2} yields
\begin{equation}
	\left\|\frac{\partial(\Pi_h\mathbf{w}-\mathbf{w})}{\partial t}\right\|_1\leq Ch\left\{t^{-\alpha}\left\|\mathbf{w}(0)\right \|_2+\left\|\frac{\partial\mathbf{w}}{\partial t}\right\|_2+\int_{0}^{t}(t-s)^{-\alpha}\left\|\frac{\partial\mathbf{w}}{\partial t}(s)\right\|_2 \, ds\right\}.
\end{equation}
Integrating the above inequality with respect to $t$, we get
\begin{equation}\label{est416}
	\begin{aligned}
		&\int_{0}^{t}\left\|\frac{\partial(\Pi_h\mathbf{w}-\mathbf{w})}{\partial t}\right\|_1 \, ds\\
		&\leq Ch\left\{\left\|\mathbf{w}(0)\right \|_2+\int_{0}^{t}\left\|\frac{\partial\mathbf{w}}{\partial t}\right\|_2 \, ds+\int_{0}^{t}\int_{0}^{s}(s-r)^{-\alpha}\|\frac{\partial\mathbf{w}}{\partial t}(r)\|_2 \,dr \,ds\right\}\\
		&\leq\ Ch\left\{\left\|\mathbf{w}(0)\right \|_2+\int_{0}^{t}\left\|\frac{\partial\mathbf{w}}{\partial t}\right\|_2 \, ds\right\}.
	\end{aligned}
\end{equation}

By following a similar routine of dual arguments as  in the proof of \cref{RVprojection}, we can  deduce the $L_2$ estimate
\begin{equation*}
	\begin{aligned}
		\int_{0}^{t}\left\|\frac{\partial(\Pi_h\mathbf{w}-\mathbf{w})}{\partial t}(s)\right\| \, ds&\leq C\left(h\int_{0}^{t}\left\|\frac{\partial(\Pi_h\mathbf{w}-\mathbf{w})}{\partial t}(s)\right\|_1 \, ds+h^2\left\|\mathbf{w}(0)\right \|_2\right)\\
		&\leq\ Ch^2\left\{\left\|\mathbf{w}(0)\right \|_2+\int_{0}^{t}\left\|\frac{\partial\mathbf{w}}{\partial t}\right\|_2 \, ds\right\}.
	\end{aligned}
\end{equation*}
This inequality, together with \cref{est416}, gives the desired estimate \cref{estRV}.
\end{proof}

\subsection{Error estimation}

By virtue of \cref{RVprojection,RVprojection_t}, we easily obtain the following error estimate for the semi-discrete finite element method: 
\begin{theorem}
	Let $\mathbf{v}$ and $\mathbf{v}_h$ be the solutions of the weak problem \cref{weak problem} and the semi-discrete scheme \cref{semidiscrete}, respectively. Then 
	\begin{equation}\label{semidiserror}
		\|\mathbf{v}-\mathbf{v}_h\|_{L^{\infty}({\bf L}^2)}+h\|\mathbf{v}-\mathbf{v}_h\|_{L^1({\bf H}^1)}\leq Ch^2\left(\left\|\mathbf{v}_0\right \|_2+\left\|\mathbf{v}\right\|_{L^\infty({\bf H}^2)}+\int_{0}^{T}\left\|\frac{\partial\mathbf{v}}{\partial t}\right\|_2 \, ds\right).
	\end{equation}
\end{theorem}
\begin{proof}
	Denote $\bm{\delta}:=\Pi_h\mathbf{v}-\mathbf{v}_h$, then
	\begin{equation}\label{etadelta}
		\mathbf{v}-\mathbf{v}_h
		=\mathbf{v}-\Pi_h\mathbf{v}+\bm{\delta}.
	\end{equation}
Owing to \cref{RVprojection}, we only need to estimate $\bm{\delta}$. From 
\cref{weak problem}, 
\cref{semidiscrete} and 
\cref{projections-RV} it follows
\begin{equation*}
	\begin{aligned}
		&\ \ \frac{1}{2}\frac{d\|\bm{\delta}\|^2}{dt}+a(\bm{\delta},\bm{\delta})=\langle\frac{\partial\bm{\delta}}{\partial t},\bm{\delta}\rangle+a(\bm{\delta},\bm{\delta})\\
		=&\ \ \int_{0}^{t}\beta(t-s)b(\bm{\delta}(s),\bm{\delta}) \, ds+\langle\frac{\partial(\Pi_h\mathbf{v}-\mathbf{v})}{\partial t},\bm{\delta}\rangle\\
		\leq&\ \ C\int_{0}^{t}(t-s)^{-\alpha}\|\bm{\delta}(s)\|_1\cdot\|\bm{\delta}(t)\|_1 \, ds+\left\|\frac{\partial(\Pi_h\mathbf{v}-\mathbf{v})}{\partial t}\right\|\cdot\|\bm{\delta}\|.
	\end{aligned}
\end{equation*}
Integrating the above inequality with respect to $t$ from $0$ to $r\in (0,T]$  and using \cref{fuzhuyinli-continuous},  we obtain
\begin{equation*}
	\begin{aligned}
		&\ \ \|\bm{\delta}(r)\|^2+\int_{0}^{r}\|\bm{\delta}(t)\|_1^2 \, dt\\
		\leq &\ \  C\left\{\|\bm{\delta}(0)\|^2+\int_{0}^{r}\left\|\frac{\partial(\Pi_h\mathbf{v}-\mathbf{v})(t)}{\partial t}\right\|\cdot\|\bm{\delta}(t)\| \, dt+\int_{0}^{r}(r-t)^{-\alpha}\int_{0}^{t}\|\bm{\delta}(s)\|_1^2 \,ds \,dt\right\}\\
		\leq &\ \  C\left\{\|\bm{\delta}(0)\|^2+\int_{0}^{r}\left\|\frac{\partial(\Pi_h\mathbf{v}-\mathbf{v})(t)}{\partial t}\right\|\cdot\|\bm{\delta}(t)\| \, dt+\int_{0}^{r}(r-t)^{-\alpha}\int_{0}^{r}\|\bm{\delta}(s)\|_1^2 \,ds \,dt\right\}.
	\end{aligned}
\end{equation*}
Thus, applying \cref{fuzhuyinli2}  we deduce that
\begin{equation*}
	\begin{aligned}
		&\|\bm{\delta}(r)\|^2+\int_{0}^{r}\|\bm{\delta}(t)\|_1^2 \, dt\leq C\left\{\|\bm{\delta}(0)\|^2+\int_{0}^{r}\left\|\frac{\partial(\Pi_h\mathbf{v}-\mathbf{v})(t)}{\partial t}\right\|\cdot\|\bm{\delta}(t)\| \, dt\right\}
	\end{aligned}
\end{equation*}
for any $r\in (0,T]$, 
which further leads to 
 $$\|\bm{\delta}\|_{L^{\infty}({\bf L}^2)}\leq C\left\{\|\bm{\delta}(0)\|+\int_{0}^{T}\left\|\frac{\partial(\Pi_h\mathbf{v}-\mathbf{v})}{\partial t}\right\| \, dt\right\}. $$
As a result, the desired estimate \cref{semidiserror} follows from \cref{etadelta,initialassu,RVprojection,RVprojection_t}.

\end{proof}

\section{Fast fully discrete finite element method }

\subsection{Fully discrete scheme}
Let $0=t_0<t_1<...<t_{\mathrm{N}}=T$ be a uniform division of the time interval $[0,T]$,  with $t_i=i\Delta t \ (i=0,1,...,N)$   and  the time step size $\Delta t:=\frac{T}{\mathrm{N}}$. 
For any function $\varphi(t)$, we set
\begin{align*}
	\varphi^n:=\varphi(t_n). 
\end{align*}

For  the time derivative  $\frac{\partial\mathbf{v}_h}{\partial t}$ in the semi-discrete scheme \cref{semidiscrete}, we apply the backward Euler difference scheme to approximate it, i.e.
$$ \frac{\partial\mathbf{v}_h}{\partial t}\approx \frac{\mathbf{v}_h^n-\mathbf{v}_h^{n-1}}{\Delta t}.$$
 Noticing    in   \cref{semidiscrete}   the   term  
 $$\int_{0}^{t}\beta(t-s)b(\mathbf{v}_h(s),\mathbf{w}_h) \, ds=b(\int_{0}^{t}\beta(t-s)\mathbf{v}_h(s)\, ds,\mathbf{w}_h), $$
 we adopt  
 the SOE approximation  \cref{SOEapp} for $\beta(t)=E_{\alpha}(-t^\alpha)$ and  evaluate  the time integral term as follows:
\begin{align*}
	&\int_{0}^{t_n}\beta(t_n-s)\mathbf{v}_h(s) \, ds \\
	\approx&\sum_{j=1}^{N_{exp}}\int_{0}^{t_n}b_je^{-\frac{a_j}{\tau_\sigma}(t_n-s)}\mathbf{v}_h(s) \, ds\\
	=&\sum_{j=1}^{N_{exp}}\left(e^{-\frac{a_j}{\tau_\sigma}\Delta t}\int_{0}^{t_{n-1}}b_je^{-\frac{a_j}{\tau_\sigma}(t_{n-1}-s)}\mathbf{v}_h(s) \, ds+\int_{t_{n-1}}^{t_n}b_je^{-\frac{a_j}{\tau_\sigma}(t_n-s)}\mathbf{v}_h(s) \, ds\right)\\
	\approx&\sum_{j=1}^{N_{exp}}e^{-\frac{a_j}{\tau_\sigma}\Delta t}\int_{0}^{t_{n-1}}b_je^{-\frac{a_j}{\tau_\sigma}(t_{n-1}-s)}\mathbf{v}_h(s) \, ds+\frac{b_j\tau_\sigma}{a_j}(1-e^{-\frac{a_j}{\tau_\sigma}\Delta t})\mathbf{v}_h^{n-1}.
\end{align*} 
Define the memory variables $H_j(\mathbf{v}_h^n)$ recursively by 
\begin{equation}\label{Memorydef}
\left\{	\begin{aligned}
		&H_j(\mathbf{v}_h^n)=e^{-\frac{a_j}{\tau_\sigma}\Delta t}H_j(\mathbf{v}_h^{n-1})+\frac{b_j\tau_\sigma}{a_j}(1-e^{-\frac{a_j}{\tau_\sigma}\Delta t})\mathbf{v}_h^{n-1}, \quad n=1,2, \cdots,N,\\
	&H_j(\mathbf{v}_h^0)=\mathbf{0},	
	\end{aligned}
	\right.
\end{equation}
then we obtain 
the following fast fully  discrete scheme: For each $n=1,...,N$, find $\mathbf{v}_h^n\in V_h$ such that 
\begin{equation}\label{fulldiscrete}
	\left\{
	\begin{array}{lll}
		\langle\frac{\mathbf{v}_h^n-\mathbf{v}_h^{n-1}}{\Delta t},\mathbf{w}_h\rangle+a(\mathbf{v}_h^n,\mathbf{w}_h)-\sum_{j=1}^{N_{exp}}b(H_j(\mathbf{v}_h^n),\mathbf{w}_h)=\langle \mathbf{F}^n,\mathbf{w}_h\rangle, \ \forall\ \mathbf{w}_h\in V_h, \\ 
		\mathbf{v}_h^0=\mathbf{v}_{0,h}.
	\end{array}
	\right.
\end{equation}

By \cref{Memorydef}  we have
\begin{align*}
	&\sum_{j=1}^{N_{exp}}b(H_j(\mathbf{v}_h^n),\mathbf{w}_h)\\
	=&\sum_{j=1}^{N_{exp}}e^{-\frac{a_j}{\tau_\sigma}\Delta t}b(H_j(\mathbf{v}_h^{n-1}),\mathbf{w}_h)+\sum_{j=1}^{N_{exp}}\frac{b_j\tau_\sigma}{a_j}(1-e^{-\frac{a_j}{\tau_\sigma}\Delta t})b(\mathbf{v}_h^{n-1},\mathbf{w}_h)\\
	& \hphantom{{}={}} \vdots \quad \text{(repeat the recursion $n-2$ times)} \\
	=&\sum_{i=0}^{n-1}\theta_{n-i}b(\mathbf{v}_h^i,\mathbf{w}_h)=\sum_{i=0}^{n-1}\theta_{n-i}b(\mathbf{v}_h^i,\mathbf{w}_h),
\end{align*}
where
\begin{equation}\label{thetadef}
	\theta_i=\sum_{j=1}^{N_{exp}}\frac{b_j\tau_\sigma}{a_j}(e^{-\frac{(i-1)\Delta ta_j}{\tau_\sigma}}-e^{-\frac{i\Delta ta_j}{\tau_\sigma}}), \quad i=1, 2,\cdots, n.
\end{equation}
Therefore, the scheme \cref{fulldiscrete} can be rewritten as follows: For each $n=1,...,N$, find $\mathbf{v}_h^n\in V_h$ such that 
\begin{equation}\label{equi-fulldiscrete}
	\left\{
	\begin{array}{lll}
		\langle\frac{\mathbf{v}_h^n-\mathbf{v}_h^{n-1}}{\Delta t},\mathbf{w}_h\rangle+a(\mathbf{v}_h^n,\mathbf{w}_h)-\sum_{i=0}^{n-1}\theta_{n-i}b(\mathbf{v}_h^i,\mathbf{w}_h)=\langle \mathbf{F}^n,\mathbf{w}_h\rangle, \ \forall\ \mathbf{w}_h\in V_h, \\ 
		\mathbf{v}_h^0=\mathbf{v}_{0,h}.
	\end{array}
	\right.
\end{equation}
It is easy to see that the above scheme 
has a unique solution for each $n$. As a result, we have the following result:
\begin{lemma}
	The fully discrete scheme \cref{fulldiscrete} admits a unique solution for $1\leq n\leq N$.
\end{lemma}

\begin{remark}
	 In the full  discretization scheme \cref{fulldiscrete}, at the $n$-th time step we need to compute $$\frac{b_j\tau_\sigma}{a_j}(1-e^{-\frac{a_j}{\tau_\sigma}\Delta t})b(\mathbf{v}_h^{n-1},\mathbf{w}_h)$$
	and store $H_j^n$ for $j=1,2,...,N_{exp}$. This means  this scheme 
	totally requires 
	$O(N_{exp}N_s)$ memory complexity and $O(N_sN_{exp}N)$ computation complexity. Here and below we denote by $\mathcal O(N_s)$  the complexity of memory and computation related to the spatial discretization.
\end{remark}

\begin{remark}
			If,  rather than handling the convolution kernel function $\beta(t-s)$ with the SOE approximation, we   directly approximate the integral term $\int_{0}^{t_n}\beta(t_n-s)b(\mathbf{v}_h(s),\mathbf{w}_h) \, ds$ 
			by   the quadrature rule
	\begin{align*}
		\int_{0}^{t_n}\beta(t_n-s)b(\mathbf{v}_h(s),\mathbf{w}_h) \, ds\approx\sum_{i=0}^{n-1}b(\mathbf{v}_h^i,\mathbf{w}_h)\int_{t_i}^{t_{i+1}}\beta(t_n-s) \, ds,
	\end{align*}
	then the resulting fully discrete scheme reads: For $1\leq n\leq N$,
	find $\mathbf{v}_h^n\in V_h$ such that 
		\begin{equation}\label{opp-fulldiscrete}
			\left\{
			\begin{array}{lll}
				\langle\frac{\mathbf{v}_h^n-\mathbf{v}_h^{n-1}}{\Delta t},\mathbf{w}_h\rangle+a(\mathbf{v}_h^n,\mathbf{w}_h)-\sum_{i=0}^{n-1}b(\mathbf{v}_h^i,\mathbf{w}_h)\displaystyle\int_{t_i}^{t_{i+1}}\beta(t_n-s) \, ds=\langle \mathbf{F}^n,\mathbf{w}_h\rangle, \\
				\phantom{CCCCCCCCCCCCCCCCCCCCCCCCCCCCCCCCCCCC} \forall\ \mathbf{w}_h\in V_h, \\ 
				\mathbf{v}_h^0=\mathbf{v}_{0,h}.
			\end{array}
			\right.
		\end{equation}
		At the $n$-th time step, this scheme always requires evaluating the integrals $\displaystyle\int_{t_i}^{t_{i+1}}\beta(t_n-s) \, ds$, which involves $n$ computations. Consequently, this scheme totally demands $\mathcal O(N_sN)$ memory complexity and $O(N_sN^2)$ computation complexity. 
\end{remark}

\begin{remark}		
		Due to the fact  $N>>N_{exp}$ (cf. \cref{rmk2.4}), the fully discrete scheme \cref{fulldiscrete} reduces  the costs of memory and computation from $O(N_sN)$ and $O(N_sN^2)$ to $O(N_{exp}N_s)$ and  $O(N_sN_{exp}N)$, respectively. Therefore, the fully discrete scheme \cref{fulldiscrete} is actually a fast scheme.
\end{remark}

\subsection{Error estimation}

We first assume that the SOE approximation is of accuracy $\epsilon>0$, i.e.
\begin{equation}\label{SOE-accuracy}
\beta(t)=\sum_{j=1}^{N_{exp}}b_je^{-\frac{a_j}{\tau_\sigma}t}+ \mathcal{O}(\epsilon),
\end{equation}
and 
define 
\begin{equation}\label{varXinotation}
	\varXi_n(\mathbf{v}):=\int_{0}^{t_n}\beta(t_n-s)\mathbf{v}(s) \, ds-\sum_{i=0}^{n-1}\theta_{n-i}\mathbf{v}(t_i)
\end{equation}
for  the  solution $\mathbf{v}(t) $ of the weak problem \cref{weak problem}.
Then we have the following result:

\begin{lemma}\label{varXiestimation}
	There exists a constant $C$, depending only on $T$, such that for $1\leq n\leq N$,
	\begin{equation}\label{5.8}
		\Delta t\sum_{k=1}^{n}\|\varXi_k(\mathbf{v})\|_{l}\leq C\left(\Delta t\int_{0}^{t_n}\left\|\frac{\partial\mathbf{v}}{\partial t}\right\|_l \, ds+T\epsilon  \int_{0}^{t_n}\left\|\frac{\partial\mathbf{v}}{\partial t}\right\|_l \, ds\right),\quad l=0,1.
	\end{equation}
\end{lemma}
\begin{proof}
	We only show the case $l=0$, since the proof for $l=1$ is similar. 
	
	Note that 
	\begin{align*}
		&\int_{0}^{t_n}\beta(t_n-s)\mathbf{v}(s) \, ds=\int_{0}^{t_n}\sum_{j=1}^{N_{exp}}b_je^{-a_j(\frac{t_n-s}{\tau_\sigma})}\mathbf{v}(s) \, ds+ \mathcal{O}(\epsilon)\int_{0}^{t_n}\mathbf{v}(s) \, ds\\
		=&\sum_{i=0}^{n-1}\sum_{j=1}^{N_{exp}}b_j\int_{t_i}^{t_{i+1}}e^{-a_j(\frac{t_n-s}{\tau_\sigma})}\mathbf{v}(s) \, ds+ \mathcal{O}(\epsilon)\int_{0}^{t_n}\mathbf{v}(s) \, ds\\
		=&\sum_{i=0}^{n-1}\sum_{j=1}^{N_{exp}}b_je^{-\frac{a_j(n-i-1)\Delta t}{\tau_\sigma}}\int_{t_i}^{t_{i+1}}e^{-a_j(\frac{t_{i+1}-s}{\tau_\sigma})}\mathbf{v}(s) \, ds+ \mathcal{O}(\epsilon)\int_{0}^{t_n}\mathbf{v}(s) \, ds.
	\end{align*}
	Recalling the definition \cref{thetadef} of $\theta_i$, we have
	\begin{align*}
		\varXi_n(\mathbf{v}) &= \int_{0}^{t_n}\beta(t_n-\tau)\mathbf{v}(s) \, ds - \sum_{i=0}^{n-1}\theta_{n-i}\mathbf{v}(t_i) \\
		&= \sum_{i=0}^{n-1} \Bigg( \sum_{j=1}^{N_{\exp}} b_j e^{-\frac{a_j(n-i-1)\Delta t}{\tau_\sigma}} 
		\int_{t_i}^{t_{i+1}} e^{-a_j\left(\frac{t_{i+1}-s}{\tau_\sigma}\right)} \left(\mathbf{v}(s)-\mathbf{v}(t_i)\right) \, ds \\
		&\phantom{=\sum_{i=0}^{n-1} \Bigg( \sum_{j=1}^{N_{\exp}}} +  \mathcal{O}(\epsilon) \int_{t_i}^{t_{i+1}} \left(\mathbf{v}(s)-\mathbf{v}(t_i)\right) \, ds \Bigg).
	\end{align*}
	Therefore, it holds
	\begin{align*}
		\left|\varXi_n(\mathbf{v})\right| 
		&\leq \sum_{i=0}^{n-1} \Bigg[ \sum_{j=1}^{N_{\exp}} b_j e^{-\frac{a_j(n-i-1)\Delta t}{\tau_\sigma}} 
		\int_{t_i}^{t_{i+1}} e^{-a_j\left(\frac{t_{i+1}-s}{\tau_\sigma}\right)} \, ds 
		\int_{t_i}^{t_{i+1}} \left|\frac{\partial \mathbf{v}}{\partial t}(r)\right| \, dr \\
		&\phantom{\leq \sum_{i=0}^{n-1} \Bigg[ \sum_{j=1}^{N_{\exp}} b_j e^{-\frac{a_j(n-i-1)\Delta t}{\tau_\sigma}}} + \mathcal{O}(\epsilon) \int_{t_i}^{t_{i+1}} \, ds \int_{t_i}^{t_{i+1}} \left|\frac{\partial \mathbf{v}}{\partial t}(r)\right| \, dr \Bigg] \\
		&= \sum_{i=0}^{n-1} \theta_{n-i} \int_{t_i}^{t_{i+1}} \left|\frac{\partial \mathbf{v}}{\partial t}(s)\right| \, ds 
		+ \Delta t  \mathcal{O}(\epsilon) \int_{0}^{t_n} \left|\frac{\partial \mathbf{v}}{\partial t}(s)\right| \, ds,
	\end{align*}
which yields 
	\begin{align*}
		\sum_{k=1}^{n}\|\varXi_k(\mathbf{v})\|
		&\leq C\left(\sum_{k=1}^{n}\sum_{i=0}^{k-1}\theta_{k-i}\int_{t_i}^{t_{i+1}}\left\|\frac{\partial\mathbf{v}}{\partial t}(s)\right\| \, ds+T\epsilon\int_{0}^{t_n}\left\|\frac{\partial\mathbf{v}}{\partial t}(s)\right\| \, ds\right)\\
		&\leq C\left(\sum_{i=0}^{n-1}\int_{t_i}^{t_{i+1}}\left\|\frac{\partial\mathbf{v}}{\partial t}(s)\right\| \, ds\sum_{k=i+1}^{n}\theta_{k-i}+T\epsilon\int_{0}^{t_n}\left\|\frac{\partial\mathbf{v}}{\partial t}(s)\right\| \, ds\right).
	\end{align*}
	Since $a_j$ and $b_j$ in \cref{SOEapp} are positive, we have
	\begin{align*}
		\sum_{k=i+1}^{n}\theta_{k-i}&=\sum_{k=i+1}^{n}\sum_{j=1}^{N_{exp}}\frac{b_j\tau_\sigma}{a_j}\left(e^{-\frac{a_j(k-i-1)\Delta t}{\tau_\sigma}}-e^{-\frac{a_j(k-i)\Delta t}{\tau_\sigma}}\right)\\
		&=\sum_{j=1}^{N_{exp}}\frac{b_j\tau_\sigma}{a_j}\sum_{k=i+1}^{n}\left(e^{-\frac{a_j(k-i-1)\Delta t}{\tau_\sigma}}-e^{-\frac{a_j(k-i)\Delta t}{\tau_\sigma}}\right)\\
		&=\sum_{j=1}^{N_{exp}}\frac{b_j\tau_\sigma}{a_j}\left(1-e^{-\frac{a_j(n-i)\Delta t}{\tau_\sigma}}\right).
	\end{align*}
As a result, we obtain 
	\begin{equation*}
		\sum_{k=1}^{n}\|\varXi_k(\mathbf{v})\|\leq C\left(\int_{0}^{t_n}\left\|\frac{\partial\mathbf{v}}{\partial t}\right\| \, ds+T\epsilon\int_{0}^{t_n}\left\|\frac{\partial\mathbf{v}}{\partial t}\right\| \, ds\right),
	\end{equation*}
	i.e. the estimate \cref{5.8} holds.
\end{proof}

By using  the above lemma, we can obtain the following error estimate for the fast fully discrete finite element method:
\begin{theorem}\label{fulldiscrete-error}
	Let $\mathbf{v}$ and $\mathbf{v}_h^n$ $(1\leq n\leq N)$ be the solutions of the weak problem \cref{weak problem} and the fully discrete scheme \cref{fulldiscrete}, respectively. Then for $n=1,2, \cdots,N$ there holds 
	\begin{equation}\label{fulldiserror}
		\begin{aligned}
			\|\mathbf{v}_h^n - \mathbf{v}(t_n)\| 
			&\leq C(h^2 + \Delta t) \left( \left\|\mathbf{v}_0\right\|_2+\left\|\mathbf{v}\right\|_{L^\infty({\bf H}^2)}+ \int_{0}^{t_n} \left( \left\|\frac{\partial \mathbf{v}}{\partial t}\right\|_2 + \left\|\frac{\partial^2 \mathbf{v}}{\partial t^2}\right\| \right) \, ds \right) \\
			&\phantom{{}\leq C(  } 
			+ CT \epsilon \left(\left\|\mathbf{v}_0\right \|_2+\int_{0}^{t_n} \left\|\frac{\partial \mathbf{v}}{\partial t}\right\|_2 \, ds\right).
		\end{aligned}
	\end{equation}
\end{theorem}
\begin{proof}
	Note that the fully discrete scheme \cref{fulldiscrete} is equivalent to \cref{equi-fulldiscrete}. 
	
	Denote $$\bm{\delta}^n:=\Pi_h\mathbf{v}^n-\mathbf{v}_h^n,\quad n=0,1,\cdots,N,$$ where $\Pi_h$ is the Ritz-Volterra projection defined by \cref{projections-RV}. Since $\mathbf{v}_h^0=\Pi_h\mathbf{v}^0=\Pi_h\mathbf{v}(0)$, we see that  
	\begin{align}\label{vh0=0} 
	\bm{\delta}^0=0.
	\end{align}
We write
	\begin{align}\label{510}
		\mathbf{v}^n-\mathbf{v}_h^n 
		=\mathbf{v}^n-\Pi_h\mathbf{v}^n+\bm{\delta}^n.
	\end{align}
	In view of \cref{RVprojection}, we only need to estimate $\bm{\delta}^n$. 
	
	From   \cref{weak problem,equi-fulldiscrete,projections-RV}, we get
	\begin{equation}\label{etaequation}
		\langle\frac{\bm{\delta}^n-\bm{\delta}^{n-1}}{\Delta t},\bm{\delta}^n\rangle+a(\bm{\delta}^n,\bm{\delta}^n)=\sum_{i=0}^{n-1}\theta_{n-i}b(\bm{\delta}^i,\bm{\delta}^n)+ R^n(\bm{\delta}^n),
	\end{equation}
	where
\begin{equation*}
		\begin{aligned}
			  R^n(\bm{\delta}^n):=-&\langle\frac{\partial\mathbf{v}}{\partial t}^n-\frac{\Pi_h\mathbf{v}^n-\Pi_h\mathbf{v}^{n-1}}{\Delta t},\bm{\delta}^n\rangle\\
			&+\int_{0}^{t_n}\beta(t_n-s)b(\Pi_h\mathbf{v}(s),\bm{\delta}^n) \, ds-\sum_{i=0}^{n-1}\theta_{n-i}b(\Pi_h\mathbf{v}^i,\bm{\delta}^n).
		\end{aligned}
	\end{equation*}
Notice that by  \cref{varXinotation} we further have 	
		\begin{equation}\label{512}
		\begin{aligned}
			  R^n(\bm{\delta}^n)
			=-&\langle\frac{\partial\mathbf{v}}{\partial t}^n-\frac{\Pi_h\mathbf{v}^n-\Pi_h\mathbf{v}^{n-1}}{\Delta t},\bm{\delta}^n\rangle +b(\varXi_n(\Pi_h\mathbf{v}),\bm{\delta}^n).
		\end{aligned}
	\end{equation}
	From \cref{a-property,etaequation,bounded-property} it follows
		\begin{equation*}
		\begin{aligned}
		&\frac{1}{2}\frac{\|\bm{\delta}^n\|^2-\|\bm{\delta}^{n-1}\|^2}{\Delta t}+\|\bm{\delta}^n\|_1^2\\
		\leq &C\left(\frac{1}{2}\frac{\|\bm{\delta}^n\|^2-\|\bm{\delta}^{n-1}\|^2}{\Delta t}+\frac{\Delta t}{2}\|\frac{\bm{\delta}^n-\bm{\delta}^{n-1}}{\Delta t}\|^2+a(\bm{\delta}^n,\bm{\delta}^n)\right)\\
		=&C\left(\langle\frac{\bm{\delta}^n-\bm{\delta}^{n-1}}{\Delta t},\bm{\delta}^n\rangle+a(\bm{\delta}^n,\bm{\delta}^n)\right)\\
		\leq& C\left(\sum_{i=0}^{n-1}\theta_{n-i}\|\bm{\delta}^i\|_1\cdot\|\bm{\delta}^n\|_1+|R^n(\bm{\delta}^n)|\right).
	\end{aligned}
	\end{equation*}
	Summing this inequality from $n=1$ to $n$ 
	and using \cref{vh0=0}, we obtain
	\begin{equation}\label{513}
		\|\bm{\delta}^n\|^2+\Delta t\sum_{k=1}^{n}\|\bm{\delta}^k\|_1^2
		\leq 
		C\Delta t\sum_{k=1}^{n}|R^k(\bm{\delta}^k)|+C\Delta t\sum_{k=1}^{n}\sum_{i=0}^{k-1}\theta_{k-i}\|\bm{\delta}^i\|_1\cdot\|\bm{\delta}^k\|_1.
	\end{equation}
	By the definition \cref{thetadef} of $\theta_i$ and the SOE approximation \cref{SOE-accuracy}, we see that
	 \begin{equation*}
		\theta_{k-i}=\sum_{j=1}^{N_{exp}}b_j\int_{t_{k-i-1}}^{t_{k-i}}e^{-\frac{a_j}{\tau_\sigma}(t_k-s)}\, ds\leq\int_{t_{k-i-1}}^{t_{k-i}}\beta(t_k-s)\, ds+\mathcal{O}(\epsilon)\Delta t,
	\end{equation*}
which, together with  \cref{513,fuzhuyinlikappa},   gives 
	\begin{align*}
		&\|\bm{\delta}^n\|^2+\Delta t\sum_{k=1}^{n}\|\bm{\delta}^k\|_1^2\\
		\leq&
		C\Delta t\sum_{k=1}^{n}|R^k(\bm{\delta}^k)| 
		+C\Delta t\sum_{k=1}^{n}\sum_{i=0}^{k-1} \int_{t_{k-i-1}}^{t_{k-i}}\left(\beta(t_k-s)+ \epsilon\right)\, ds\|\bm{\delta}^i\|_1\cdot\|\bm{\delta}^k\|_1\\
		\leq&
		C\Delta t\sum_{k=1}^{n}|R^k(\bm{\delta}^k)|  
		+\frac {\Delta t}2\sum_{k=1}^{n}\|\bm{\delta}^k\|_1^2
		+C\Delta t\sum_{k=0}^{n-1}\int_{t_{n-k-1}}^{t_{n-k}} \left(\beta(t_n-s)+\epsilon\right)\, ds  \sum_{i=0}^{k}\|\bm{\delta}^i\|_1^2.
	\end{align*}
Thus, due to   \cref{513,fuzhuyinlikappa,vh0=0}  we further have
	\begin{align*}
		&\|\bm{\delta}^n\|^2+\Delta t\sum_{k=1}^{n}\|\bm{\delta}^k\|_1^2\\
		\leq& 
		C\Delta t\sum_{k=1}^{n}|R^n(\bm{\delta}^k)|+C \sum_{k=0}^{n-1}\int_{t_{n-k-1}}^{t_{n-k}} \left(\beta(t_n-s)+\epsilon\right)\, ds \left(\Delta t\sum_{i=1}^{k}\|\bm{\delta}^i\|_1^2\right),
	\end{align*}
which plus    \cref{discretegronwall,MLbound}, yields
	\begin{equation}\label{5.14}
		\begin{aligned}
			\|\bm{\delta}^n\|^2+\Delta t\sum_{k=1}^{n}\|\bm{\delta}^k\|_1^2\leq C
			\Delta t\sum_{k=1}^{n}|R^k(\bm{\delta}^k)|. 
		\end{aligned}
	\end{equation}
	In the following let us  estimate 
	$|R^k(\bm{\delta}^k)|$. To this end, 
we write 
$$R^k(\bm{\delta}^k)=\sum_{j=1}^{4}R^k_j(\bm{\delta}^k),$$ where
	\begin{align*}
		&R^k_1(\bm{\delta}^k):=\langle-\frac{\partial\mathbf{v}}{\partial t}^k+\frac{\mathbf{v}^k-\mathbf{v}^{k-1}}{\Delta t},\bm{\delta}^k\rangle,& &R^k_2(\bm{\delta}^k):=\langle\frac{(\Pi_h\mathbf{v}-\mathbf{v})^k-(\Pi_h\mathbf{v}-\mathbf{v})^{k-1}}{\Delta t},\bm{\delta}^k\rangle,\\
		&R^k_3(\bm{\delta}^k):=b(\varXi_k(\mathbf{v}),\bm{\delta}^k), &&R^k_4(\bm{\delta}^k):=b(\varXi_k(\Pi_h\mathbf{v}-\mathbf{v}),\bm{\delta}^k).
	\end{align*}
	For $R^k_1(\bm{\delta}^k)$ we   have
	\begin{align*}
		\Delta t\sum_{k=1}^{n}|R^k_1(\bm{\delta}^k)|&=\Delta t\sum_{k=1}^{n}\left|\langle-\frac{\partial\mathbf{v}}{\partial t}^k+\frac{\mathbf{v}^k-\mathbf{v}^{k-1}}{\Delta t},\bm{\delta}^k\rangle\right|\\ 
		&=\Delta t\sum_{k=1}^{n}\left|\langle\int_{t_{k-1}}^{t_k}(t_k-s)\frac{\partial^2\mathbf{v}(s)}{\partial t^2}\ ds,\bm{\delta}^k\rangle\right|\\ 
		&\leq \Delta t\sum_{k=1}^{n}\int_{t_{k-1}}^{t_k}\left\|\frac{\partial^2\mathbf{v}}{\partial t^2}\right\| \, ds  \cdot \|\bm{\delta}^k\|.
	\end{align*}
	For  $R^k_2(\bm{\delta}^k)$ we easily get
	\begin{align*}
		\Delta t\sum_{k=1}^{n}|R^k_2(\bm{\delta}^k)|&=  \sum_{k=1}^{n}\left|\langle (\Pi_h\mathbf{v}-\mathbf{v})^k-(\Pi_h\mathbf{v}-\mathbf{v})^{k-1},\bm{\delta}^k\rangle\right|\\
		&=  \sum_{k=1}^{n}\left|\langle \int_{t_{k-1}}^{t_k} \frac{\partial(\Pi_h\mathbf{v}-\mathbf{v})(s)}{\partial t}  \, ds,\bm{\delta}^k\rangle\right|\\
		&\leq\sum_{k=1}^{n}\int_{t_{k-1}}^{t_k}\left\|\frac{\partial(\Pi_h\mathbf{v}-\mathbf{v})(s)}{\partial t}\right\| \, ds  \cdot \|\bm{\delta}^k\|.
	\end{align*}
	For $R^k_3(\bm{\delta}^k)$,  
	by integration by parts and \cref{operatordefB}
	we obtain
	\begin{align*}
		\Delta t\sum_{k=1}^{n}|R^k_3(\bm{\delta}^k)|&=\Delta t\sum_{k=1}^{n}\left|b(\varXi_n(\mathbf{v}),\bm{\delta}^k)\right|
		=\Delta t\sum_{k=1}^{n}\left|\langle\varXi_k(\B \mathbf{v}),\bm{\delta}^k\rangle\right|\\
		&\leq \Delta t\sum_{k=1}^{n}\|\varXi_k(\B \mathbf{v})\| \cdot \|\bm{\delta}^k\|.
	\end{align*}
	For $R^k_4(\bm{\delta}^k)$, from  \cref{bounded-property} and the inverse inequality \cref{inverseassum} it follows 
	\begin{align*}
		\Delta t\sum_{k=1}^{n}|R^k_4(\bm{\delta}^k)|&=\Delta t\sum_{k=1}^{n}|b(\varXi_k(\Pi_h\mathbf{v}-\mathbf{v}),\bm{\delta}^k)|\\
		&\leq C\Delta t\sum_{k=1}^{n}\|\varXi_k(\Pi_h\mathbf{v}-\mathbf{v})\|_1 \cdot \|\bm{\delta}^k\|_1\\
		&\leq Ch^{-1}\Delta t \sum_{k=1}^{n}\|\varXi_k(\Pi_h\mathbf{v}-\mathbf{v})\|_1 \cdot \|\bm{\delta}^k\|.
	\end{align*}
	
Combining the above four estimates 
and \cref{5.14}
implies
\begin{equation}\label{5.15}
	\begin{aligned}
		\|\bm{\delta}^n\|^2  \leq C \sum_{k=1}^{n} \tilde{R}_k   \|\bm{\delta}^k\| 
	\leq C S_n \sum_{k=1}^{n} \tilde{R}_k , 
	\end{aligned}
\end{equation}
where
\begin{equation*}
\left\{
\begin{aligned}
	\tilde{R}_k:=&\Delta t\int_{t_{k-1}}^{t_k}\left\|\frac{\partial^2\mathbf{v}}{\partial t^2}\right\| \, ds + \int_{t_{k-1}}^{t_k}\left\|\frac{\partial(\Pi_h\mathbf{v}-\mathbf{v})}{\partial t}\right\| \, ds\\
	&\quad + \Delta t\|\varXi_k(\B \mathbf{v})\| 
	+  Ch^{-1} \Delta t\|\varXi_k(\Pi_h\mathbf{v}-\mathbf{v})\|_1,\\
	S_n:=&\max_{1 \leq k \leq n}\|\bm{\delta}^k\|.
\end{aligned}
\right.
\end{equation*}
%
Hence, 
we have 
\begin{equation*}\label{5.16}
\begin{aligned}
	\|\bm{\delta}^n\|\leq & S_n\leq C \sum_{k=1}^{n} \tilde{R}_k\\
	\leq &C\left( \Delta t\int_{0}^{t_n}\left\|\frac{\partial^2\mathbf{v}}{\partial t^2}\right\| \, ds + \int_{0}^{t_n}\left\|\frac{\partial(\Pi_h\mathbf{v}-\mathbf{v})}{\partial t}\right\| \, ds\right.\\
		&\left.\quad + \Delta t\sum_{k=1}^{n}\|\varXi_k(\B \mathbf{v})\| + Ch^{-1}\Delta t\sum_{k=1}^{n} \|\varXi_k(\Pi_h\mathbf{v}-\mathbf{v})\|_1\right).
\end{aligned}
\end{equation*}
This estimate, together with 
 \cref{initialassu,RVprojection,RVprojection_t,varXiestimation}, yields  
    the desired  result \cref{fulldiserror}.

\end{proof}

\section{Numerical experiments}
In this section, we provide two-dimensional numerical experiments to verify the performance of the fast scheme \cref{fulldiscrete}.
 All computations were performed using MATLAB 2023a on a Windows 10 workstation equipped with a 3.40~GHz processor and 32~GB RAM. 
 
 We consider the viscoelastic model~\eqref{model} in   $\Omega = [0,1] \times [0,1]$ with the following parameters: final simulation time $ T = 1 $, relaxation time $ \tau_\sigma = 1/2 $, retardation time $ \tau_\varepsilon = 1 $, material density $ \rho = 1 $, and SOE parameters $q=10$ and $\epsilon=\frac{\Delta t}{10}$. The elasticity tensors are characterized by the parameters $\mu_C = \lambda_C = 1$ for $\mathbb{C}$, and $\mu_D = 1$, $\lambda_D = 2$ for $\mathbb{D}$. We use both triangular and square meshes for the spatial region $\Omega$, and uniform grids for the time region $[0,T]$. The spatial meshes for $\Omega$ are illustrated in \cref{fig:mesh}.
 
  To assess numerical accuracy, we compute the $L^2$-error for the velocity field at final time $T$:
\begin{equation*}
	\text{Error} := \| \mathbf{v}_h(T) - \mathbf{v}(T) \|_{{\bf L}^2(\Omega)}.
\end{equation*}
From  Theorem~\ref{fulldiscrete-error}  we know that the theoretical accuracy of this $L^2$-error  is 
\begin{equation}
	\text{Error}=  \mathcal{O} (h^2 + \Delta t+\epsilon).
	\label{eq:velocity_error}
\end{equation}

	
	
	\begin{figure}[H] 
	\centering
	\begin{minipage}{0.49\textwidth}
		\centering
		\hspace*{-3mm}
		\includegraphics[width=0.51\linewidth]{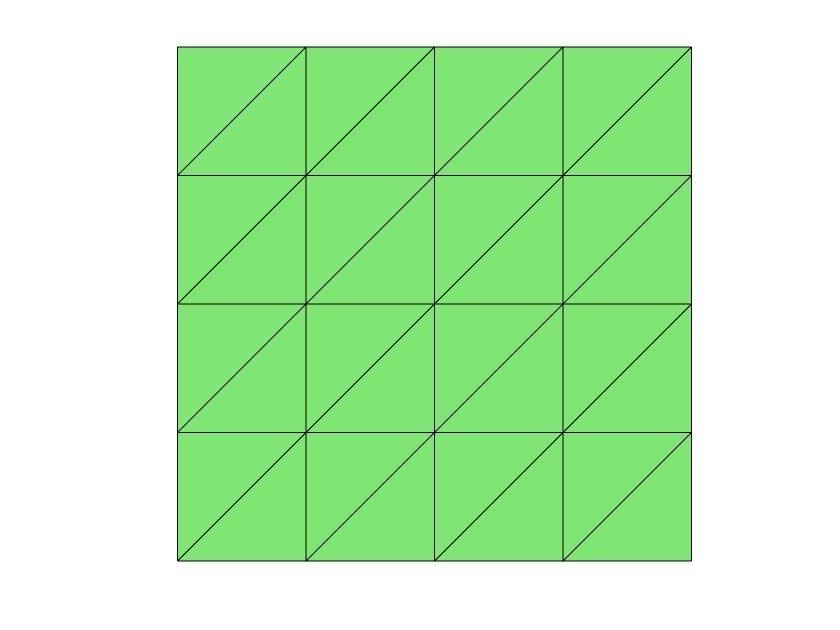}
		\hspace*{-4mm}
		\includegraphics[width=0.51\linewidth]{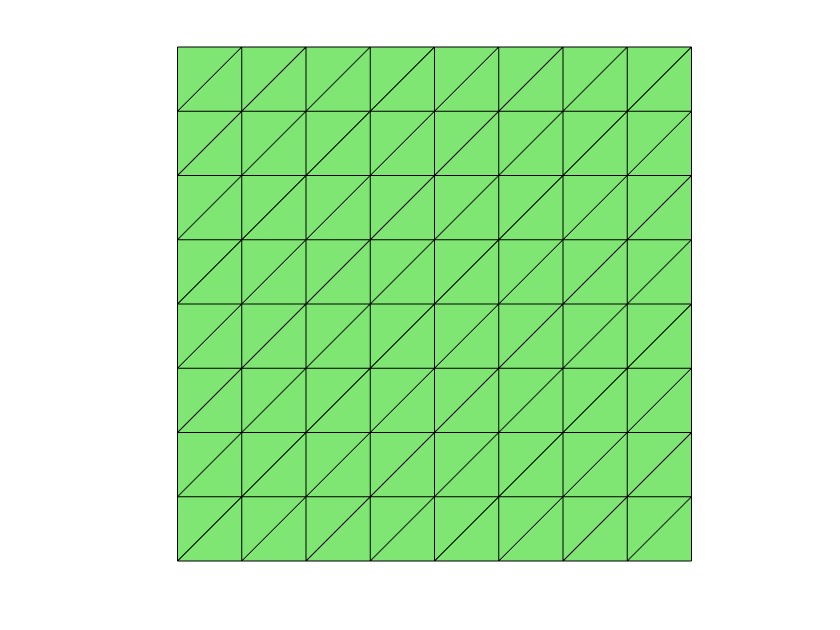}
		\hspace*{-3mm}
		
		\vspace{2mm}
		\text{Triangular meshes}
	\end{minipage}
	\hfill
	\begin{minipage}{0.49\textwidth}
		\centering
		\hspace*{-3mm}
		\includegraphics[width=0.51\linewidth]{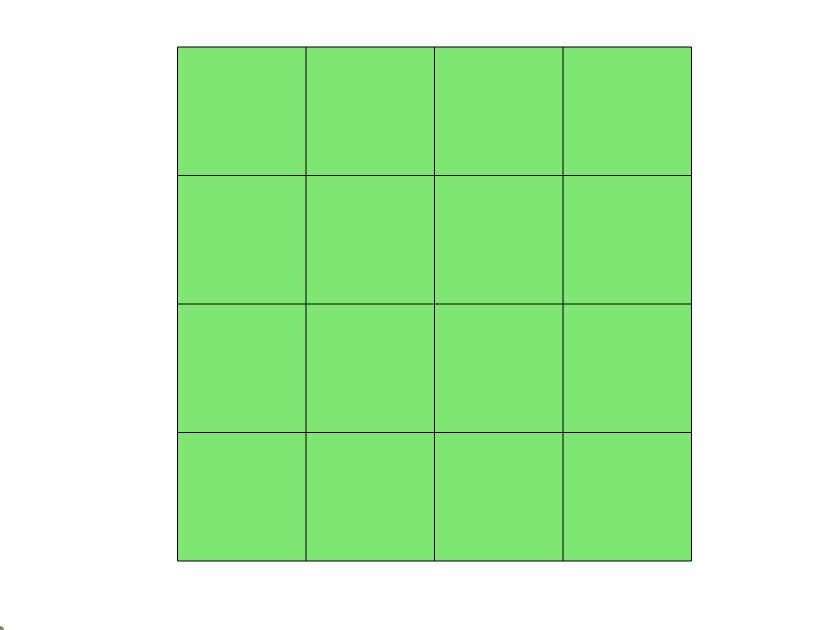}
		\hspace*{-4mm}
		\includegraphics[width=0.51\linewidth]{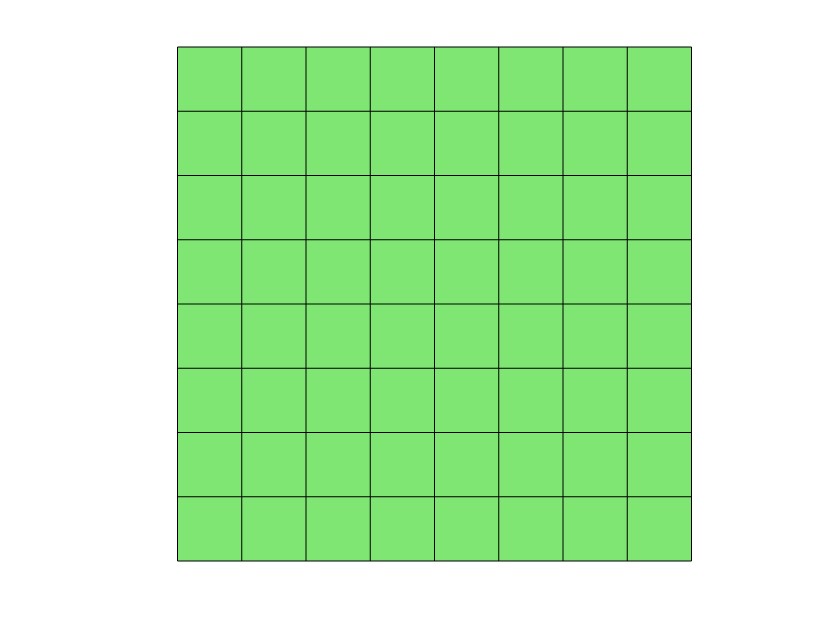}
		\hspace*{-3mm}
		
		\vspace{2mm} 
		\text{Square meshes}
	\end{minipage}
	\caption{ The domain $\Omega$: two types of  meshes}
	\label{fig:mesh}
	\end{figure}

\begin{example}\label{example3}
	The exact velocity field $\mathbf{v}(x,t)$ is given by
	\begin{align*}
		\mathbf{v}(x,t)=\left(
		\begin{array}{l}
			e^{-t}\sin(\pi x)\sin(\pi y)\\
			e^{-t}\sin(\pi x)\sin(\pi y)
		\end{array}
		\right),
	\end{align*}
	Numerical results   are shown in \cref{ex3:combined_error_h,ex3:combined_error_t}, 
	and the wall time and   memory cost of the computation are demonstrated in \cref{difftimeandmemory-2d-trian,difftimeandmemory-2d}.
\end{example}

\begin{table}[H]
	\centering
	\small
	\renewcommand{\arraystretch}{1.1} 
	\setlength{\tabcolsep}{4pt} 
	\renewcommand{\arraystretch}{1.2}
	\begin{tabular}{c c *{6}{>{$}c<{$}}}
		\toprule
		\multirow{2}{*}{Mesh} 
		& \multirow{2}{*}{$\frac{h}{\sqrt{2}}$} 
		& \multicolumn{2}{c}{$\alpha = 0.3$} 
		& \multicolumn{2}{c}{$\alpha = 0.5$} 
		& \multicolumn{2}{c}{$\alpha = 0.8$} \\
		\cmidrule(lr){3-4} \cmidrule(lr){5-6} \cmidrule(lr){7-8}
		& & \text{Error} & \text{Order} 
		& \text{Error} & \text{Order} 
		& \text{Error} & \text{Order} \\
		\midrule
		\multirow{5}{*}{\rotatebox{90}{Triangular}\hspace{1mm}} 
		& 1/4 & 4.90\times10^{-2} & -- & 4.93\times10^{-2} & -- & 4.98\times10^{-2} & -- \\
		& 1/8 & 9.67\times10^{-3} & 2.34 & 9.79\times10^{-3} & 2.33 & 9.99\times10^{-3} & 2.32 \\
		& 1/16 & 2.17\times10^{-3} & 2.16 & 2.20\times10^{-3} & 2.15 & 2.25\times10^{-3} & 2.15 \\
		& 1/32 & 5.13\times10^{-4} & 2.08 & 5.33\times10^{-4} & 2.05 & 5.43\times10^{-4} & 2.05 \\
		& 1/64 & 1.13\times10^{-4} & 2.18 & 1.31\times10^{-4} & 2.03 & 1.34\times10^{-4} & 2.02 \\
		\midrule
		\multirow{5}{*}{\rotatebox{90}{Square}\hspace{1mm}} 
		& 1/4 & 1.82\times10^{-2} & -- & 1.88\times10^{-2} & -- & 1.96\times10^{-2} & -- \\
		& 1/8 & 4.58\times10^{-3} & 1.99 & 4.73\times10^{-3} & 1.99 & 4.98\times10^{-3} & 1.97 \\
		& 1/16 & 1.18\times10^{-3} & 1.96 & 1.22\times10^{-3} & 1.96 & 1.27\times10^{-3} & 1.97 \\
		& 1/32 & 2.86\times10^{-4} & 1.99 & 3.07\times10^{-4} & 1.99 & 3.19\times10^{-4} & 2.00 \\
		& 1/64 & 5.74\times10^{-5} & 2.02 & 7.57\times10^{-5} & 2.02 & 7.91\times10^{-5} & 2.00 \\
		\bottomrule
	\end{tabular}
	\caption{Spatial accuracy test for \cref{example3} on different meshes with $\Delta t = \frac{h^2}{2}$.}
	\label{ex3:combined_error_h}
\end{table}

\begin{table}[H]
	\centering
	\small
	\renewcommand{\arraystretch}{1.1} 
	\setlength{\tabcolsep}{4pt} 
	\renewcommand{\arraystretch}{1.2}
	\begin{tabular}{c c *{6}{>{$}c<{$}}}
		\toprule
		\multirow{2}{*}{Mesh} 
		& \multirow{2}{*}{$\Delta t$} 
		& \multicolumn{2}{c}{$\alpha = 0.3$} 
		& \multicolumn{2}{c}{$\alpha = 0.5$} 
		& \multicolumn{2}{c}{$\alpha = 0.8$} \\
		\cmidrule(lr){3-4} \cmidrule(lr){5-6} \cmidrule(lr){7-8}
		& & \text{Error} & \text{Order} 
		& \text{Error} & \text{Order} 
		& \text{Error} & \text{Order} \\
		\midrule
		\multirow{5}{*}{\rotatebox{90}{Triangular}\hspace{1mm}} 
		&1/5 & 2.57\times10^{-2} & -- & 3.43\times10^{-2} & -- & 4.35\times10^{-2} & -- \\
		&1/10 & 1.17\times10^{-2} & 1.13 & 1.53\times10^{-2} & 1.17 & 1.90\times10^{-2} & 1.19 \\
		&1/20 & 5.58\times10^{-3} & 1.07 & 7.19\times10^{-3} & 1.08 & 8.93\times10^{-3} & 1.09 \\
		&1/40 & 2.86\times10^{-3} & 0.97 & 3.63\times10^{-3} & 0.99& 4.42\times10^{-3} & 1.02 \\
		&1/80 & 1.44\times10^{-3} & 0.99 & 1.83\times10^{-3} & 0.99 & 2.22\times10^{-3} & 1.00 \\
		\midrule
		\multirow{5}{*}{\rotatebox{90}{Square}\hspace{1mm}} 
		&1/5 & 2.56\times10^{-2} & -- & 3.43\times10^{-2} & -- & 4.34\times10^{-2} & -- \\
		&1/10 & 1.17\times10^{-2} & 1.13 & 1.52\times10^{-2} & 1.17 & 1.90\times10^{-2} & 1.20 \\
		&1/20 & 5.54\times10^{-3} & 1.08 & 7.15\times10^{-3} & 1.09 & 8.89\times10^{-3} & 1.09 \\
		&1/40 & 2.81\times10^{-3} & 0.98 & 3.58\times10^{-3} & 1.00& 4.38\times10^{-3} & 1.02 \\
		&1/80 & 1.41\times10^{-3} & 1.00 & 1.79\times10^{-3} & 1.00 & 2.18\times10^{-3} & 1.01 \\
		\bottomrule
	\end{tabular}
	\caption{Temporal accuracy test for  \cref{example3} with $h = \frac{\sqrt{2}}{64}$ on different meshes.}
	\label{ex3:combined_error_t}
\end{table}

\begin{figure}[H]
	\centering 
	\subfigure{
		\includegraphics[width=.4\linewidth]{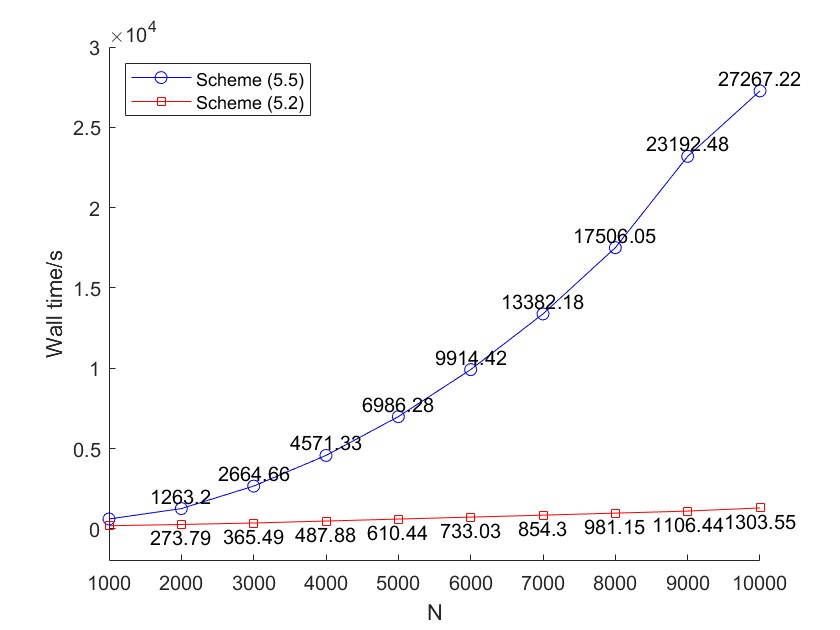}
		\label{difftime-2d-trian}
	}
	\subfigure{
		\includegraphics[width=.4\linewidth]{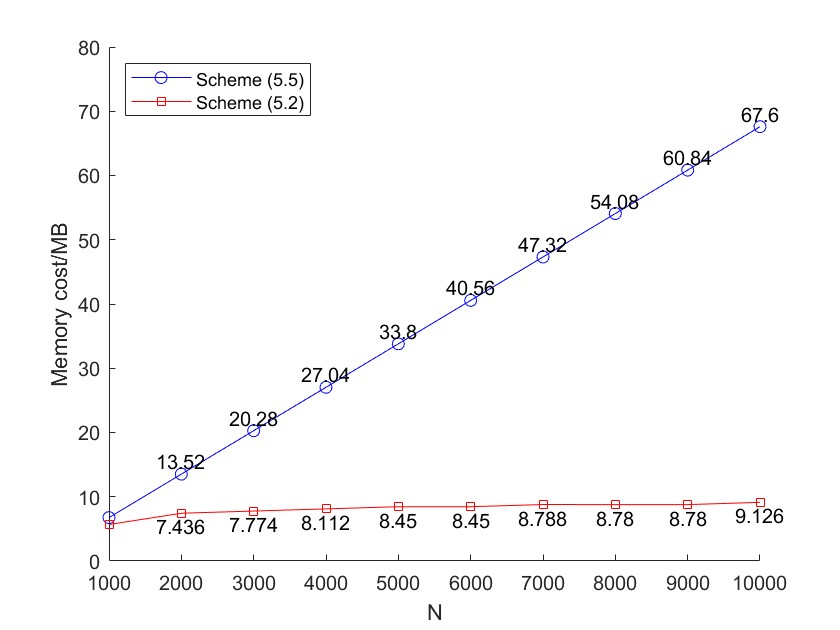}
		\label{diffmemory-2d-trian}
	}
	\caption{ Comparison of wall time and memory cost between the fast scheme \cref{fulldiscrete} and the full discretization  \cref{opp-fulldiscrete} for \cref{example3} on triangular meshes at different time steps $t_N$:  $h=\frac{\sqrt{2}}{64}$ and $\alpha=0.5$. }
	\label{difftimeandmemory-2d-trian}	
\end{figure}

%

\begin{figure}[H]
	\centering 
	\subfigure{
		\includegraphics[width=.4\linewidth]{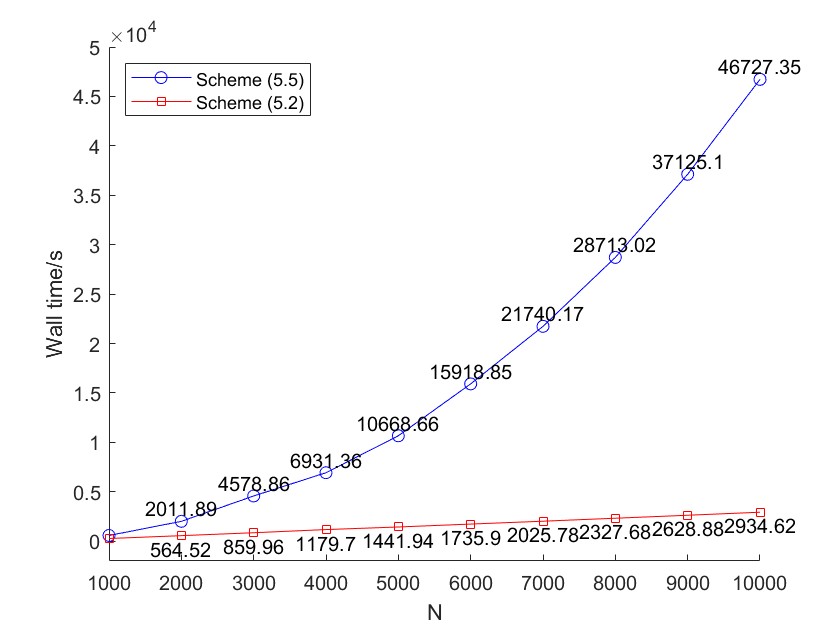}
		\label{difftime-2d}
	}
	\subfigure{
		\includegraphics[width=.4\linewidth]{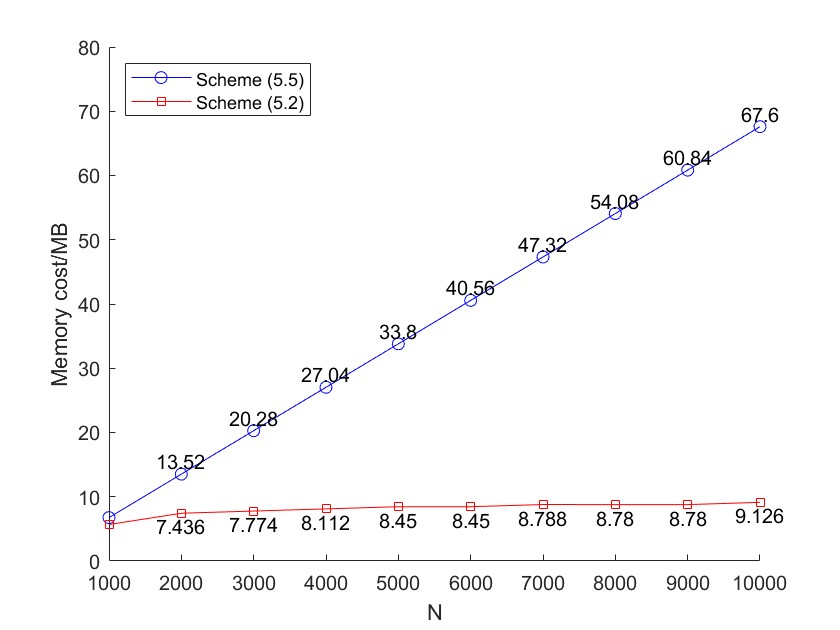}
		\label{diffmemory-2d}
	}
	\caption{ Comparison of wall time and memory cost between the fast scheme \cref{fulldiscrete} and the full discretization  \cref{opp-fulldiscrete} for \cref{example3} on square meshes at different time steps $t_N$:  $h=\frac{\sqrt{2}}{64}$ and $\alpha=0.5$.}
	\label{difftimeandmemory-2d}	
\end{figure}

\begin{example}\label{example4}
	The exact velocity field $\mathbf{v}(x,t)$ is given by
	\begin{align*}
		\mathbf{v}(x,t)=\left(
		\begin{array}{l}
			e^{-t}(x^4-2x^3+x^2)(4y^3-6y^2+2y)\\
			e^{-t}(y^4-2y^3+y^2)(4x^3-6x^2+2x)
		\end{array}
		\right),
	\end{align*}
	Numerical results   are shown in \cref{ex4:combined_error_h,ex4:combined_error_t}.
\end{example}

\begin{table}[H]
	\centering
	\small
	\renewcommand{\arraystretch}{1.1} 
	\setlength{\tabcolsep}{4pt} 
	\renewcommand{\arraystretch}{1.2}
	\begin{tabular}{c c *{6}{>{$}c<{$}}}
		\toprule
		\multirow{2}{*}{Mesh} 
		& \multirow{2}{*}{$\frac{h}{\sqrt{2}}$} 
		& \multicolumn{2}{c}{$\alpha = 0.3$} 
		& \multicolumn{2}{c}{$\alpha = 0.5$} 
		& \multicolumn{2}{c}{$\alpha = 0.8$} \\
		\cmidrule(lr){3-4} \cmidrule(lr){5-6} \cmidrule(lr){7-8}
		& & \text{Error} & \text{Order} 
		& \text{Error} & \text{Order} 
		& \text{Error} & \text{Order} \\
		\midrule
		\multirow{5}{*}{\rotatebox{90}{Triangular}\hspace{1mm}} 
		&1/4 & 1.07\times10^{-3} & -- & 1.07\times10^{-3} & -- & 1.07\times10^{-3} & -- \\
		&1/8 & 2.62\times10^{-4} & 2.03 & 2.63\times10^{-4} & 2.03 & 2.65\times10^{-4} & 2.02 \\
		&1/16 & 6.52\times10^{-5} & 2.00 & 6.55\times10^{-5} & 2.00 & 6.60\times10^{-5} & 2.00 \\
		&1/32 & 1.63\times10^{-5} & 2.00 & 1.65\times10^{-5} & 1.99 & 1.66\times10^{-5} & 1.99 \\
		&1/64 & 3.96\times10^{-6} & 2.04 & 4.11\times10^{-6} & 2.00 & 4.14\times10^{-6} & 2.00 \\
		\midrule
		\multirow{5}{*}{\rotatebox{90}{Square}\hspace{1mm}} 
		&1/4 & 6.38\times10^{-4} & -- & 6.39\times10^{-4} & -- & 6.42\times10^{-4} & -- \\
		&1/8 & 1.57\times10^{-4} & 2.02 & 1.58\times10^{-4} & 2.02 & 1.59\times10^{-4} & 2.02 \\
		&1/16 & 3.91\times10^{-5} & 2.01 & 3.92\times10^{-5} & 2.01 & 3.94\times10^{-5} & 2.01 \\
		&1/32 & 9.71\times10^{-6} & 2.00 & 9.79\times10^{-6} & 2.00 & 9.84\times10^{-6} & 2.00 \\
		&1/64 & 2.38\times10^{-6} & 2.00 & 2.44\times10^{-6} & 2.00 & 2.46\times10^{-6} & 2.00 \\
		\bottomrule
	\end{tabular}
	\caption{Spatial accuracy test  for \cref{example4}  on different meshes with $\Delta t = \frac{h^2}{2}$.}
	\label{ex4:combined_error_h}
\end{table}

\begin{table}[H]
	\centering
	\small
	\renewcommand{\arraystretch}{1.1} 
	\setlength{\tabcolsep}{4pt} 
	\renewcommand{\arraystretch}{1.2}
	\begin{tabular}{c c *{6}{>{$}c<{$}}}
		\toprule
		\multirow{2}{*}{Mesh} 
		& \multirow{2}{*}{$\Delta t$} 
		& \multicolumn{2}{c}{$\alpha = 0.3$} 
		& \multicolumn{2}{c}{$\alpha = 0.5$} 
		& \multicolumn{2}{c}{$\alpha = 0.8$} \\
		\cmidrule(lr){3-4} \cmidrule(lr){5-6} \cmidrule(lr){7-8}
		& & \text{Error} & \text{Order} 
		& \text{Error} & \text{Order} 
		& \text{Error} & \text{Order} \\
		\midrule
		\multirow{5}{*}{\rotatebox{90}{Triangular}\hspace{1mm}} 
		&1/5 & 2.59\times10^{-4} & -- & 3.52\times10^{-4} & -- & 4.56\times10^{-4} & -- \\
		&1/10 & 1.29\times10^{-4} & 1.01 & 1.67\times10^{-4} & 1.08 & 2.07\times10^{-4} & 1.14 \\
		&1/20 & 6.20\times10^{-5} & 1.05 & 7.94\times10^{-5} & 1.07 & 9.80\times10^{-5} & 1.08 \\
		&1/40 & 3.09\times10^{-5} & 1.01 & 3.91\times10^{-5} & 1.02& 4.85\times10^{-5} & 1.02 \\
		&1/80 & 1.54\times10^{-5} & 1.00 & 1.95\times10^{-5} & 1.00 & 4.42\times10^{-5} & 1.00 \\
		\midrule
		\multirow{5}{*}{\rotatebox{90}{Square}\hspace{1mm}} 
		&1/5 & 2.78\times10^{-4} & -- & 3.71\times10^{-4} & -- & 4.65\times10^{-4} & -- \\
		&1/10 & 1.26\times10^{-4} & 1.14 & 1.64\times10^{-4} & 1.17 & 2.05\times10^{-4} & 1.19 \\
		&1/20 & 5.99\times10^{-5} & 1.08 & 7.73\times10^{-5} & 1.09 & 9.60\times10^{-5} & 1.09 \\
		&1/40 & 3.05\times10^{-5} & 0.98 & 3.88\times10^{-5} & 1.00 & 4.74\times10^{-5} & 1.02 \\
		&1/80 & 1.54\times10^{-5} & 0.99 & 1.95\times10^{-5} & 0.99 & 2.37\times10^{-5} & 1.00 \\
		\bottomrule
	\end{tabular}
	\caption{Temporal accuracy test for \cref{example4}  on different meshes with $h = \frac{\sqrt{2}}{64}$.}
	\label{ex4:combined_error_t}
\end{table}

From all the numerical results we have the following observations:
\begin{itemize}
	\item From   \cref{ex3:combined_error_h,ex4:combined_error_h} we see that the spatial convergence rate of    $\| \mathbf{v}_h(T) - \mathbf{v}(T) \|_{{\bf L}^2(\Omega)}$ for the   fast scheme is second order, which is 
	consistent with the theoretical accuracy \cref{eq:velocity_error} due to the setting $\Delta t = \frac{h^2}{2}$ and $\epsilon=\frac{\Delta t}{10}$.
	
	
	\item From \cref{ex3:combined_error_t,ex4:combined_error_t}  we see that  the temporal  convergence rate of    $\| \mathbf{v}_h(T) - \mathbf{v}(T) \|_{{\bf L}^2(\Omega)}$   is first order, which is also
	conformable to the theoretical accuracy \cref{eq:velocity_error} with $\epsilon=\frac{\Delta t}{10}$ and a sufficiently small spatial mesh size $h = \frac{\sqrt{2}}{64}$.
	
	
	\item From \cref{difftimeandmemory-2d-trian,difftimeandmemory-2d} we see that,  compared with  those of the scheme \cref{opp-fulldiscrete}, 
	both the wall time and  memory  cost of the  fast scheme \cref{fulldiscrete} are significantly  reduced, especially  as $N$ becomes large. 
\end{itemize}
 
\section{Conclusions}
We have shown the existence, uniqueness and regularity of the weak solution to the velocity type integro-differential equation of  fractional viscoelastic  wave propagation.  For the spatial semi-discretization using low order conforming finite elements and the fast full discretization using the backward-Euler scheme and the SOE approximation for  the temporal discretization, we have derived rigorous  error estimates. 
Numerical results have confirmed  the theoretical analysis and demonstrated the   efficiency of the developed fast scheme.

\bibliographystyle{plain}
\normalem
\bibliography{fraviscoelastic}

\end{document}